\documentclass[12pt, oneside, psamsfonts]{amsart}

\newif\ifPDF
\ifx\pdfoutput\undefined\PDFfalse
\else \ifnum \pdfoutput > 0 \PDFtrue
        \else \PDFfalse
        \fi
\fi

\usepackage[centertags]{amsmath}
\usepackage{amsfonts}
\usepackage{mathrsfs}
\usepackage{textcomp}
\usepackage{amssymb}
\usepackage{amsthm}
\usepackage{newlfont}
\usepackage[all]{xy}

\ifPDF
  \usepackage[pdftex]{color, graphicx}
  \usepackage[pdftex, bookmarks, colorlinks]{hyperref}
  \hypersetup{colorlinks=false}

\else
  \usepackage{color}
  \usepackage[dvips]{graphicx}
  \usepackage[dvips]{hyperref}
\fi

\usepackage[scale=0.8]{geometry}

\usepackage[pagewise, mathlines, displaymath]{lineno}

\newtheorem{thm}{Theorem}[section]
\newtheorem{cor}[thm]{Corollary}
\newtheorem{lem}[thm]{Lemma}

\newtheorem{prop}[thm]{Proposition}
\theoremstyle{definition}
\newtheorem{defn}[thm]{Definition}
\theoremstyle{remark}
\newtheorem{rem}[thm]{Remark}

\numberwithin{equation}{section}

\newcommand{\norm}[1]{\left\Vert#1\right\Vert}

\newcommand{\Real}{\mathbb R}
\newcommand{\Int}{\mathbb Z}
\newcommand{\Comp}{\mathbb C}

\newcommand{\eps}{\varepsilon}

\newcommand{\Kzero}{\mathrm{K}_0}
\newcommand{\Kone}{\mathrm{K}_1}

\newcommand{\tr}{\mathrm{tr}}

\newcommand{\aff}{\mathrm{Aff}}

\title[Classification of KK-contractible C*-algebras]{The classification of simple separable KK-contractible C*-algebras with finite nuclear dimension}
\author{George A. Elliott}
\address{Department of Mathematics, University of Toronto, Toronto, Ontario, Canada~\ M5S 2E4}
\email{elliott@math.toronto.edu}

\author{Zhuang Niu}
\address{Department of Mathematics, University of Wyoming, Laramie, WY 82071, USA}
\email{zniu@uwyo.edu}

\begin{document}

\begin{abstract}
The class of simple separable KK-contractible (KK-equivalent to $\{0\}$) C*-algebras which have finite nuclear dimension is shown to classified by the Elliott invariant. In particular, the class of C*-algebras $A\otimes \mathcal W$ is classifiable, where $A$ is a simple separable C*-algebra with finite nuclear dimension and $\mathcal W$ is the simple inductive limit of Razak algebras with unique trace, which is bounded (see \cite{Razak-W} and \cite{Jacelon-W}).
\end{abstract}

\maketitle

\section{Introduction}

The classification of separable simple unital UCT C*-algebras with finite nuclear dimension has been completed (\cite{GLN-TAS}, \cite{EGLN-ASH}, \cite{EN-K0-Z}, \cite{EGLN-DR}, \cite{TWW-QD}). In this paper, we study the classification for non-unital C*-algebras. In particular, the following classification theorem is obtained:

\theoremstyle{plain}
\newtheorem*{Theorem}{Theorem}
\begin{Theorem}[Theorem \ref{clas-thm}]
The class of simple separable KK-contractible C*-algebras with finite nuclear dimension and with non-zero traces is classified by the invariant $(\mathrm T^+A, \Sigma A)$. Any C*-algebra $A$ in this class is a simple inductive limits of Razak algebras.
\end{Theorem}

In the theorem, $\mathrm{T}^+A$ is the cone of densely defined lower semicontinuous traces of $A$ with the topology of pointwise convergence, and $\Sigma A$ is the compact subset consisting of the traces with norm at most one.

Consider the C*-algebra $\mathcal W$, the simple inductive limit of Razak algebras with unique trace which is bounded (see \cite{Razak-W} and \cite{Jacelon-W}). It is KK-contractible, and hence $A\otimes \mathcal W$ is KK-contractible for any nuclear C*-algebra $A$. Thus if $A$ has finite nuclear dimension, the C*-algebra $A\otimes\mathcal W$ has finite nuclear dimension as well (if the Toms-Winter conjecture holds, then $A\otimes\mathcal W$ has finite nuclear dimension for all simple nuclear C*-algebra $A$), and hence is classifiable:

\theoremstyle{plain}
\newtheorem*{Corollary}{Corollary}
\begin{Corollary}[Corollary \ref{cor-WW}]
Let $A$ be a simple separable C*-algebra with finite nuclear dimension. Then the C*-algebra $A\otimes \mathcal W$ is classifiable. In particular, $\mathcal W\otimes \mathcal W \cong \mathcal W$.
\end{Corollary}

The strategy to the classification theorem is similar to the unital case. Assume $A$ is a simple C*-algebra with trivial $\Kzero$-group, has finite nuclear dimension, and the norm function on $\mathrm{T}^+A$ is continuous (assume $\mathrm{T}^+A\neq \{0\}$), then $A\otimes Q$ has the following tracial approximation structure (Theorem \ref{div-decp}):

For any finite set $\mathcal F\subset A$, any $\eps>0$ and any $n\in \mathbb N$, there are projections $q, p_1, p_2, ..., p_n \in \tilde{A}\otimes Q$, $S\subset p_1(A\otimes Q)p_1$, $S$ is a Razak algebra, such that
\begin{enumerate}
\item $q+p_1+p_2+\cdots+p_n =1_{\tilde{A}\otimes Q}$, $q$ is Murray-von Neumann equivalent to $p_i$, $i=1, 2, ..., n$,
\item $\|[p_i, a\otimes 1_Q]\| < \eps$, $a\in\mathcal F$,
\item $p_iap_i\in_{\eps}v_iSv^*_i$, $a\in\mathcal F$, where $v_i\in \tilde{A}\otimes Q$, $i=1, 2, ..., n$, are the partial isometries  satisfying $v^*_iv_i=p_1$ and $v_iv_i^*=p_i$.
\end{enumerate}
With tracial approximation structure and a stable uniqueness theorem (Corollary \ref{stable-uniq-trace}), the classification theorem is then obtained.

\section{The reduction class $\mathcal R$}

Let $A$ be a C*-algebra. Denote by $\mathrm{T}^+A$ to be the cone of densely defined lower semicontinuous traces of $A$ with the topology of pointwise convergence, and denote by $\Sigma A$ the compact convex subset consisting of the traces with norm at most one. 

\begin{defn}
A C*-algebra $A$ is said to be in the reduction class, denoted by $\mathcal R$, if $A$ is simple, separable, and non-unital, if any element of $\mathrm{T}^+A$ is bounded, and if, moreover, the affine function on $\mathrm{T}^+A$ arising from the Cuntz class of a strict positive element is continuous (equivalently, the norm function on $\mathrm{T}^+A$ is continuous). (The case $\mathrm{T}^+A=\{0\}$ vacuously satisfies the last two conditions, but let us exclude this case.)

\end{defn}

\begin{lem}
Let $A$ be a nonzero simple separable Jiang-Su stable C*-algebra. Then $A$ is stably isomorphic to a C*-algebra in the reduction class. 
\end{lem}
\begin{proof}
Choose a non-zero positive element $a$ in the Pedersen ideal of $A$. By Lemma 6.5 of \cite{ESR-Cuntz}, there exists a positive element $c$ of $A$ such that the Cuntz class $[c]$ gives rise to the same affine function on $\mathrm{T}^+A$ as $a$, and in additional is not the continuous of a projection. In particular, this affine function is the norm function on $\mathrm{T}^+A$ for the non-unital hereditary sub-C*-algebra $B$ of $A$, which on the one hand is in the reduction class, and on the other hand by Brown's theorem is stably isomorphic to $A$. (We are using Proposition 4.7 of \cite{CPtrace} which says that the restriction map is an isomorphism between the topological cones $\mathrm{T}^+A$ and $\mathrm{T}^+B$.)
\end{proof}

%

The following two lemmas shows that, in particular,  a C*-algebra in the reduction class always has an approximate unit $(e_n)$ such that  $e_n$ is self-adjoint with spectrum $[0, 1]$.

\begin{lem}\label{pre-full-sp}
Let $A$ be a separable non-unital C*-algebra. Then, for any finite set $\mathcal{F}\subset A$ and any $\eps>0$, there is a subalgebra $B\subset A$ such that $\mathcal F\subset_\eps B$, but $B^\perp \neq \{0\}$. 
\end{lem}
\begin{proof}
Pick a strictly positive element, and denote it by by $h$. Then $0$ is a limit point of $\mathrm{sp}(h)$, as otherwise $A$ must be unital. Then, for the given $\mathcal F$ and $\eps$, there is $\delta>0$ such that 
$$\mathcal F\subset_\eps \overline{g_\delta(h)Ag_\delta(h)}=:B,$$ where
$$
g_\delta(t)=\left\{
\begin{array}{ll}
0, & t\in[0, \delta], \\
(t-\delta)/(1-\delta), & t\in [\delta, 1].
\end{array}
\right.
$$
Since $0$ is a limit point of $\mathrm{sp}(h)$, it is clear that $B^\perp\neq \{0\}$, as desired.
\end{proof}

\begin{lem}\label{full-sp}
Let $A$ be a separable simple non-unital non-elementary C*-algebra. Then there is an approximate unit $(e_n)$ such that each $e_n$, $n=1, 2, ...$, is positive of norm one, and has full spectrum.
\end{lem}
\begin{proof}

Let $\mathcal F\subset A$ be a finite set, and let $\eps>0$ be arbitrary. By Lemma \ref{pre-full-sp}, there is a subalgebra $B\subset A$ such that $\mathcal F\subset_\eps B$, but $B^\perp \neq \{0\}$. Since $B$ is simple and non-elementary, the hereditary subalgebra $B^\perp \neq \{0\}$ is also non-elementary, and hence there is $e'\in B^\perp$ such that $\mathrm{sp}(e')=[0, 1]$. Pick a norm-one positive element $e''\in B$ such that $\|e'' a - a\|<2\eps$ for all $a\in\mathcal F$. Then the element $e:=e'\oplus e''$ satisfies
$$\|e a - a\|<4\eps\quad\mathrm{and}\quad \mathrm{sp}(e)=[0, 1],$$
as desired.
\end{proof}

\section{Trace model algebras for stably projectionless C*-algebras}

\begin{defn}[\cite{point-line}]
Let $E$ and $F$ be finite dimensional C*-algebras, and let $\phi_0, \phi_1: E \to F$ be two homomorphisms (not necessarily unital). Then the C*-algebra
$$A(E, F, \phi_0, \phi_1)=\{(e, f)\in E \oplus \mathrm{C}([0, 1], F): f(0)=\phi_0(e),\ f(1)=\phi_1(e)\}$$
is called an Elliott-Thomsen algebra or a point--line algebra.
\end{defn}

\begin{defn}[\cite{Razak-W}]
Let $k, n\in\mathbb{N}$. Consider homomorphisms $\phi_0, \phi_1: \mathrm{M}_k(\Comp) \to \mathrm{M}_{k(n+1)}(\Comp)$ defined by 
$$\phi_0(a)=a\otimes 1_n\quad\mathrm{and}\quad \phi_1(a) = a\otimes 1_{n+1}.$$ Then the C*-algebra
$$S(k, n)=A(\mathrm{M}_k(\Comp), \mathrm{M}_{k(n+1)}(\Comp), \phi_0, \phi_1)$$
is called a Razak algebra. Let us also call a direct sum of such C*-algebras a Razak algebra, and denote this class of C*-algebras by $\mathcal R_{z}$.
\end{defn}

As shown in \cite{point-line}, any simple partially ordered group together with a pairing with a topological cone with a base a Choquet simples can be realized as the value of the invariant of a simple inductive limits of Elliott-Thomsen algebras. For the propose of this paper, we shall consider the simple inductive systems $(S_i, \iota_i)$ such that each $S_i$ a $Q$-stable point--line algebras and
\begin{enumerate}
\item\label{model-cond-1} All traces of $S$ are bounded and all traces of $S_i$ are bounded (this is automatic).
\item\label{model-cond-2} For any $\eps>0$, there is a sufficiently large $i$ such that if one writes each $S_i=S\{E, F, \phi_0, \phi_1\}$ where $E=\bigoplus_p Q$, $F=\bigoplus_l Q$, and $\phi_0, \phi_1: E\to F$, and writes 
$$
\left( 
\begin{array}{cccc}
[\phi_0]_{1, 1} & [\phi_0]_{1, 2} & \cdots & [\phi_0]_{1, p} \\
{[\phi_0]}_{2, 1} & [\phi_0]_{2, 2} & \cdots & [\phi_0]_{2, p} \\
\vdots & \vdots & \ddots & \vdots \\
{[\phi_0]}_{l, 1} & [\phi_0]_{l, 2} & \cdots & [\phi_0]_{l, p} 
\end{array}
\right)
\quad
\mathrm{and}
\quad
\left( 
\begin{array}{cccc}
[\phi_1]_{1, 1} & [\phi_1]_{1, 2} & \cdots & [\phi_1]_{1, p} \\
{[\phi_1]}_{2, 1} & [\phi_1]_{2, 2} & \cdots & [\phi_1]_{2, p} \\
\vdots & \vdots & \ddots & \vdots \\
{[\phi_1]}_{l, 1} & [\phi_1]_{l, 2} & \cdots & [\phi_1]_{l, p} 
\end{array}
\right)
$$
as the multiplicity matrices of $\phi_0$ and $\phi_1$, respectively, then
$$ [\phi_k]_{i, 1} + [\phi_k]_{i, 2} + \cdots + [\phi_k]_{i, p}> 1-\eps,\quad k=0, 1,\ i=1, 2, ..., l.$$
\item\label{model-cond-3} For each $\eps>0$, there is $e(\eps)\in S_i^+$ for sufficiently large $i$ such that $\|e(\eps)\|=1$ and 
\begin{equation}\label{trace-control-1}
\tau(e(\eps))>1-\eps,\quad \tau\in\mathrm{T}^+S_j,\ \|\tau\|=1, \ j\geq i
\end{equation} 
and
\begin{equation}\label{trace-control-2}
\tau(e(\eps))>1-\eps,\quad \tau\in\mathrm{T}^+S,\ \|\tau\|=1.
\end{equation}
\end{enumerate}

\begin{rem}
It follows from \eqref{trace-control-2} that the norm function $\tau\mapsto \|\tau\|$ is continuous on $\mathrm{T}S$. Hence the tracial states of $S$ form a compact base of $\mathrm{T}S$.
\end{rem}

Then, any compact metrizable Choquet simplex can be realized as the simplex of tracial states of a simple inductive limit of Razak algebras and the inductive system satisfies the conditions above:


\begin{thm}\label{model-trace}
Let $\Delta$ be a compact metrizable Chouqet simplex. Then there is a simple inductive limit C*-algebra $S=\varinjlim (S_i, \phi_{i})$, where each $S_i\in\mathcal{R}_z$, such that the tracial states of $S$ form a compact simplex which is homeomorphic to $\Delta$, and moreover,  $S\otimes Q$ satisfies Conditions (\ref{model-cond-1}), (\ref{model-cond-2}), and (\ref{model-cond-3}).
\end{thm}

\begin{proof}
Consider the compact metrizable Chouqet simplex $\Xi$. Consider the strictly positive $f \in \aff C$ with $f(\Delta)=\{1\}$, and pick a (unital) simple AF algebra $B= \overline{\bigcup_{i=1}^\infty B_i}$ such that $\mathrm{T}^+B \cong C$ and the isomorphism sends the tracial state with norm one to $\Xi$.

Choose a decreasing sequence of positive numbers $(\eps_i)$ such that $\sum_{i=1}^\infty \eps_i<\infty$. Also choose a dense sequence $(x_i)$ in $(0, 1)$.

Write $B_i=\bigoplus_{j=1}^{w_1} \mathrm{M}_{k_{1, j}}(\Comp)$, and put 
$$S_1=\bigoplus_{j=1}^{w_1} S(k_{1, j}, 1).$$ Write $B_i=\bigoplus_{j=1}^{w_i} \mathrm{M}_{k_{i, j}}(\Comp)$. Assume that the $i$th algebra 
$$S_i=\bigoplus_{j=1}^{w_i} S(k_{i, j}, n_{i, j})$$
is already defined for some $n_{i, j}$. Let us construct $S_{i+1}$ and a homomorphism $\phi_{i}: S_i \to S_{i+1}$.

Write $B_{i+1}=\bigoplus_{j'=1}^{w_{j+1}} \mathrm{M}_{k_{i+1, j'}}(\Comp)$. Write $m_{j, j'}$, $1\leq j \leq w_i$, $1\leq j' \leq w_{i+1}$, the multiplicities of the embedding of $B_i$ into $B_{i+1}$. Since $B$ is simple, one may assume that the multiplicities are sufficiently large such that
\begin{equation}
m_{j, j'}-n_{i, j}-1 > \frac{n_{i, j}}{\eps_i},\quad j'=1, 2, ..., w_{i+1},\ j=1, 2, ..., w_i.
\end{equation}
Then choose an integer $l$ sufficiently large such that
$$\frac{l}{l+1+\sum_{j=1}^{w_i}(m_{j, j'}-1)}>1-\eps_i,\quad j'=1, 2, ..., w_{i+1}.$$

For each $j'=1, 2, ..., w_{i+1}$, define a homomorphism $\theta_{j'}: S_i \to \mathrm{M}_{k_{i+1, j'}}(\Comp)$ by
$$(f_1, f_2, ..., f_{w_i}) \mapsto 
\left(
\begin{array}{ccc}
\left(
\begin{array}{cc}
\underbrace{
\begin{array}{ccc}
f_1(\ast_1) &  &   \\
 & \ddots &  \\
 & & f_1(\ast_1)
\end{array}}_{m_{1, j'}-n_{i, 1}-1} & \\
 & f_1(x_i)
 \end{array}
\right) & & \\
& \ddots & \\
& &
\left(
\begin{array}{cc}
\underbrace{
\begin{array}{ccc}
f_{w_i}(\ast_{w_i}) &  &   \\
 & \ddots &  \\
 & & f_{w_i}(\ast_{w_i})
\end{array}}_{m_{w_i, j'}-n_{i, w_i} - 1} & \\
 & f_{w_i}(x_i)
 \end{array}
\right)
\end{array}
\right).
$$

For each $j=1, 2, ..., w_i$, $j'=1, 2, ..., w_{i+1}$, define $$\omega_{j, j'}: S(k_{i, j}, n_{i, j}) \to \mathrm{C}([0, 1], \mathrm{M}_{k_{i+1, j'}(m_{j, j'}-1)}(\Comp))$$ by sending $f_j$ to the function 
$$t \mapsto 
\underbrace{
\left(
\begin{array}{ccc}
\left(
\begin{array}{ccc}
\underbrace{
\begin{array}{ccc}
f_j(\ast_j) &  &   \\
 & \ddots &  \\
 & & f_j(\ast_j)
\end{array}}_{m_{j, j'}-2n_{i, j}-2} & & \\
& f_1(t) & \\
 & & f_1(x_i)
 \end{array}
\right) & & \\
& \ddots & \\
& &
\left(
\begin{array}{ccc}
\underbrace{
\begin{array}{ccc}
f_{j}(\ast_{j}) &  &   \\
 & \ddots &  \\
 & & f_{j}(\ast_{j})
\end{array}}_{m_{j, j'}-2n_{i, j} - 2} & & \\
& f_{j}(t) & \\
& & f_{j}(x_i)
 \end{array}
\right)
\end{array}
\right)}_{m_{j, j'}-1},
$$
and put
$$\omega_{j'}:=\bigoplus_{j=1}^{w_i} \omega_{j, j'}: S_i \to \bigoplus_{j=1}^{w_i} \mathrm{C}([0, 1], \mathrm{M}_{k_{i+1, j'}(m_{j, j'}-1)}(\Comp))$$

For each $j'=1, 2, ..., w_{i+1}$, define a homomorphism $\upsilon_{j'}: S_i \to \mathrm{C}([0, 1], \mathrm{M}_{k_{i+1, j'}}(\Comp))$ by sending $(f_1, f_2, ..., f_{w_i})$ to the function
$$
\left(
\begin{array}{ccc}
\left(
\begin{array}{cc}
\underbrace{
\begin{array}{ccc}
f_1(\ast_1) &  &   \\
 & \ddots &  \\
 & & f_1(\ast_1)
\end{array}}_{m_{1, j'}-n_{i, 1}-1} & \\
 & f_1(x_i+t1-tx_i)
 \end{array}
\right) & & \\
& \ddots & \\
& &
\left(
\begin{array}{cc}
\underbrace{
\begin{array}{ccc}
f_{w_i}(\ast_{w_i}) &  &   \\
 & \ddots &  \\
 & & f_{w_i}(\ast_{w_i})
\end{array}}_{m_{w_i, j'}-n_{i, w_i} - 1} & \\
 & f_{w_i}(x_i+t-tx_i)
 \end{array}
\right)
\end{array}
\right).
$$

Then, consider $$S_{i+1}=\bigoplus_{j'=1}^{w_{i+1}} S(k_{i+1, j'}, n_{i+1, j'}),$$
where $$n_{i+1, j'}=l+\sum_{j=1}^{w_i}(m_{j, j'}-1),\quad j'=1, 2, ..., w_{i+1},$$ and consider the homomorphism $\phi_{i}: S_i \to S_{i+1}$ defined by 
$$\phi_{i}= \bigoplus_{j'=1}^{w_{i+1}} U_{j'}^*(\underbrace{\theta_{j'}\oplus\cdots\oplus\theta_{j'}}_l \oplus \omega_{j'} \oplus \upsilon_{j'})U_{j'},$$
where $U_{j'}\in \mathrm{C}([0, 1], \mathrm{M}_{k_{i+1, j'}(l+1+\sum_{j=1}^{w_i}(m_{j, j'}-1))}(\Comp))$ is a unitary satisfying $U_{j'}(0)=1$ and
$$
\begin{array}{cl}
&U^*_{j'}(1)((\underbrace{\theta_{j'}\oplus\cdots\oplus\theta_{j'}}_l \oplus \omega_{j'} \oplus \upsilon_{j'})(f_1, f_2, ..., f_{w_i}))(1)U_{j'}(1)\\
=&
\left(
\begin{array}{cc}
\underbrace{
\begin{array}{ccc}
\left(
\begin{array}{cc}
\underbrace{
\begin{array}{ccc}
f_1(\ast_1) &  &   \\
 & \ddots &  \\
 & & f_1(\ast_1)
\end{array}}_{m_{1, j'}-n_{i, 1}-1} & \\
 & f_1(x_i)
 \end{array}
\right) & & \\
& \ddots & \\
& &
\left(
\begin{array}{cc}
\underbrace{
\begin{array}{ccc}
f_{w_i}(\ast_{w_i}) &  &   \\
 & \ddots &  \\
 & & f_{w_i}(\ast_{w_i})
\end{array}}_{m_{w_i, j'}-n_{i, w_i} - 1} & \\
 & f_{w_i}(x_i)
 \end{array}
\right)
\end{array}
}_{l+\sum_{j=1}^{w_i}(m_{j, j'}-1)} & \\
& 0_{k_{i+1, j'}}
\end{array}
\right).
\end{array}
$$
Then a direct calculation shows that inductive system $(S_i, \phi_i)$ satisfies the conclusion of the theorem.
\end{proof}

\section{The Affine Space}

Denote by $\aff\mathrm{T}^+A$ the continuous affine maps from $\mathrm{T}^+A$ to $\Real$. Set
$$\|a\|_\infty = \sup\{|a(\tau)|: \tau\in\Sigma A \},\quad a\in \aff\mathrm{T}^+A.$$
Then, if all traces of $A$ are bounded, $(\aff\mathrm{T}^+A, \|\cdot\|_\infty)$ is a (ordered) Banach space.

Define the scale of $\aff\mathrm{T}A$ to be the closed convex set $$\aff_1\mathrm{T}^+A:=\overline{\{\hat{a}\in\aff\mathrm{T}^+A: a\in A^+,\ \|a\|\leq 1\}}.$$ Note that $\|a\|_\infty\leq 1$ if $a\in \aff_1\mathrm{T}^+A$.  

\begin{lem}
Let $A, B$ be C*-algebras, and let $\phi: A\to B$ be a homomorphism. Then $\phi_{\mathrm T}: \aff\mathrm T^+A \to \aff\mathrm T^+B$ is a positive linear contraction, and $\phi_{\mathrm T}(\aff_1\mathrm{T}^+A) \subseteq \aff_1\mathrm{T}^+B$.
\end{lem}

\begin{lem}
Let $\tau\in\mathrm{T}^+A$. If $a(\tau)\leq 1$ for all $a\in\aff_1\mathrm{T}^+A$, then $\|\tau\|\leq 1$ (in other words, $\tau\in\Sigma A$).
\end{lem}
\begin{proof}
It follows from the assumption that $\tau(a)\leq 1$, $a\in A^+$ with $\|a\|\leq 1$. Then $\norm{\tau}=\sup\{\tau(a): a\in A^+, \|a\|\leq 1\}\leq 1$.
\end{proof}

Since $\|\hat{a}\|_\infty\leq 1$ for all $a\in A^+$ with $\|a\|\leq 1$, one has
\begin{cor}
For any C*-algebra $A$, one has 
$$\textrm{$\tau\in\Sigma A$ if and only if $|a(\tau)|\leq 1$ for all $a\in \aff\mathrm T^+A$ with $\|a\|_\infty\leq 1$}.$$
\end{cor}

Let $(A_i, \iota_i)$ be an inductive system with $\iota_i$ injective. Assume that all traces of $A$ are bounded, and assume that all traces of each $A_i$ are bounded. Then, one has the inductive system
$$
\xymatrix{
(\aff \mathrm T^+A_1, \|\cdot\|_{1, \infty}) \ar[r]^{(\iota_1)_*} & (\aff \mathrm T^+A_2, \|\cdot\|_{2, \infty}) \ar[r]^-{(\iota_2)_*}  & \cdots \ar[r] & (\aff \mathrm T^+A, \|\cdot\|_{\infty})
}.
$$
The dual system is
$$
\xymatrix{
(\mathrm T^+A_1, \Sigma A_1) & (\mathrm T^+{A_2}, \Sigma A_2) \ar[l]_{(\iota_1)^*}  & \cdots \ar[l]_-{(\iota_2)^*}  & (\mathrm T^+A, \Sigma A) \ar[l]
}.
$$


Recall that
\begin{lem}[Lemma 2.8 of \cite{EN-K0-Z}]\label{FDF}
Let $\Delta$ be a compact metrizable Choquet simplex. Then, for any finite subset $\mathcal F\subseteq\aff(\Delta)$ and any $\eps>0$, there exist $m\in\mathbb N$ and unital (pointwise) positive linear maps $\varrho$ and $\theta$,
\begin{displaymath}
\xymatrix{
\aff(\Delta) \ar[r]^-{\varrho} & \Real^m \ar[r]^-{\theta} & \aff(\Delta),
}
\end{displaymath} 
where the unit of $\Real^m$ is $(1, ..., 1)$, 
such that
$$\norm{\theta(\varrho(f))-f}_\infty<\eps,\quad f\in\mathcal F.$$
\end{lem}

\begin{lem}\label{tr-lifting}
Let $(A_n, \iota_n)$ be an inductive system of C*-algebras with $\phi_n$ injective, and assume that all traces of $\varinjlim A_n$ are bounded. Fix a compact base $\Delta\subset \mathrm{T}^+A$. Then, for any finite set $\mathcal F\subseteq \aff\mathrm{T}A$ and any $\eps>0$, there exists $$\gamma_n: \aff\mathrm{T}^+A \to \aff\mathrm{T}^+A_n$$ for sufficiently large $n$ such that 
$$ | (\iota_{n, \infty}\circ\gamma_n(a))(\tau) - a(\tau) | < \eps,\quad \tau\in\Delta.$$
\end{lem}

\begin{proof}
By Lemma \ref{FDF}, one only has to approximately lift a positive continuous affine map $\gamma: \Real^m \to \aff\mathrm{T}^+A$ to a map $\gamma_n: \Real^m \to \aff\mathrm{T}^+A_n$ for given finite subset $\mathcal F\subset \Real^n$ and $\eps>0$. But since $\Real^m$ is finite dimensional, the lifting can be easily constructed.
\end{proof}

\begin{lem}\label{trace-lifting}
Let $S=\varinjlim(S_n, \iota_n)$ be a model algebra as in Theorem \ref{model-trace}. Then, for any finite set $\mathcal F\subseteq \aff\mathrm{T}^+S$ and any $\eps>0$, there exists $$\gamma_m: \aff\mathrm{T}^+S \to \aff\mathrm{T}^+S_m$$ for sufficiently large $m$ such that 
$$ | a(\gamma_m^*\circ\iota^*_{m, \infty}(\tau)) - a(\tau) | < \eps,\quad a\in\mathcal F,\ \|\tau\|=1,$$
and 
$$1-\eps< \|\gamma^*_m(\tau)\| < 1+\eps,\quad \tau \in \mathrm{T}^+S_m,\ \|\tau\|=1. $$
%
\end{lem}
\begin{proof}
Fix an arbitrary $\eps>0$ for the time being. It follows from \eqref{trace-control-1} and \eqref{trace-control-2} that there is $e(\eps)\in S_n$ such that
\begin{equation}\label{tr-c-01}
\tau(e(\eps))>1-\eps,\quad \tau\in\mathrm{T}(S_m),\ \|\tau\|=1, \ m\geq n,
\end{equation}
and
\begin{equation}\label{tr-c-02}
\tau(e(\eps))>1-\eps,\quad \tau\in\mathrm{T}(S),\ \|\tau\|=1.
\end{equation}

Applying Lemma \ref{tr-lifting}, one has $\gamma_n: \aff\mathrm{T}^+S\mapsto\aff\mathrm{T}^+S_n$ such that 
\begin{equation}\label{tr-approx-01}
 | a(\gamma_n^*\circ\iota^*_{n, \infty}(\tau)) - a(\tau) | < \eps,\quad a\in\mathcal F\cup\{e(\eps)\},\ \tau\in\mathrm{T}^+A,\  \|\tau\|=1.
 \end{equation}

Consider the induced diagram
\begin{displaymath}
\xymatrix{
 & & \mathrm{T}^+(S) \ar@{=}[d] \\
\mathrm{T}^+S_n \ar[urr]^{\gamma_n^*} & \ar[l]^{\iota^*_{n, m}} \mathrm{T}^+S_m &\ar[l]^{\iota^*_{m, \infty}} \mathrm{T}^+S.
}
\end{displaymath}
With $m$ sufficiently large, one has that for any $\tau\in\mathrm{T}^+ A_m$ with $\|\tau\|=1$, there is $\tau'\in\mathrm{T}^+A$ such that
\begin{equation}\label{pull-trace-1}
|\tau(e(\eps)) - \tau'(e(\eps))| < \eps.
\end{equation}
and
\begin{equation}\label{pull-trace-2}
|\tau(\gamma_n(e(\eps))) - \tau'(\gamma_n(e(\eps)))| < \eps.
\end{equation}
By \eqref{tr-c-01}, $$1\geq \tau(e(\eps))>1-\eps,$$ and then  by \eqref{pull-trace-1}, one has 
\begin{equation}\label{trace-bound}
1+\eps > \tau'(e(\eps))>1-2\eps
\end{equation}
and therefore by \eqref{tr-c-02},
$$1-2\eps<\|\tau'\|< \frac{1+ \eps}{1-\eps}.$$
Then, by \eqref{tr-approx-01},
$$|((\iota_{n, \infty}\circ\gamma_n)(e(\eps)))(\tau') - \tau'(e(\eps))|<\eps\frac{1+\eps}{1-\eps},$$
and then by \eqref{trace-bound}
$$1-2\eps-\eps\frac{1+\eps}{1-\eps} < ((\iota_{n, \infty}\circ\gamma_n)(e(\eps)))(\tau') < 1+\eps + \eps\frac{1+\eps}{1-\eps}.$$
By \eqref{pull-trace-2}, 
$$1-3\eps-\eps\frac{1+\eps}{1-\eps} < ((\iota_{n, m}\circ\gamma_n)(e(\eps)))(\tau) < 1+2\eps + \eps\frac{1+\eps}{1-\eps},$$
and therefore, it follows from \eqref{tr-c-02} that
$$1-3\eps-\eps\frac{1+\eps}{1-\eps} <\|(\gamma_n^*\circ \iota^*_{n, m})(\tau)\|< (1+2\eps + \eps\frac{1+\eps}{1-\eps})\frac{1}{1-\eps}.$$
Then, with sufficiently small $\eps$, the map $\gamma_m:=\iota_{n, m}\circ \gamma_n$ satisfies the conclusion of the lemma.
\end{proof}

%
%
%
%
%

\section{A stable uniqueness theorem}
The main results (Corollary \ref{uniq-Q} and Corollary \ref{uniq-Q-path}) in this section are probably well known to experts; they will be used to show the trace factorization theorem (Theorem \ref{TFA}). A more general stable uniqueness theorem will be proved in Section \ref{section-stable-uniq} including the case that the codomain algebra is not necessarily unital. (If the codomain algebra $B$ is not $\sigma$-unital, as if $B=\prod B_n/\bigoplus B_n$ for nonunital C*-algebras $B_n$, one will have troubles for the picture of $\mathrm{KK}(A, B)$. This difficuty is handled in Section \ref{section-stable-uniq} by considering a certain sub-C*-algebra of $B$.)

\begin{defn}
Let $A$ be a C*-algebra, and let $L: A^1_+\setminus\{0\} \to \mathbb N$ be a map. Then a homomorphism $\phi: A \to B$, where $B$ is a unital C*-algebra, is $L$-full if for any $a\in A^+\setminus\{0\}$, there are contractions $x_1, x_2, ..., x_{L(a)}$ in $B$ such that
$$1_B=x^*_1\phi(a)x_1 + x^*_2\phi(a)x_2 + \cdots + x^*_{L(a)}\phi(a)x_{L(a)}.$$
\end{defn}

\begin{defn}
Let $A$ be a C*-algebra. Denote by 
$$\mathrm{Proj}_\infty(A)=\{p\in\mathrm{M}_n(A): \textrm{$p$ is a projection, $n\in\mathbb{N}$}\},$$
and define
$$\underline{\mathrm{Proj}}_\infty(A)=\bigcup_{k=0}^\infty \mathrm{Proj}_\infty(A\otimes C_k),$$
where $C_k$ is a commutative C*-algebra satisfying 
$$\Kzero(C_k) = \Int/k\Int\quad \mathrm{and}\quad \Kone(C_k) = \{0\}.$$
Also define
$$\mathrm{Proj}_\infty(\widetilde{A})\cap\Kzero(A)=\{p\in \mathrm{Proj}_\infty(\widetilde{A}): \textrm{$p$ is not equivalent to a projection in $\mathrm{M}_\infty(\Comp 1_{\widetilde{A}})$}\}.$$

\end{defn}

The following theorem is well known.
\begin{thm}\label{uniq-hom-1}
Let $A$ be a unital separable exact C*-algebra that satisfies the UCT. Let $L: A^1_+\setminus\{0\} \to \mathbb N$ be a map.

Let $\mathcal F\subset A$ be a finite subset, and let $\eps>0$. Then there exists $n\in\mathbb N$ such that, for any admissible codomain algebra $D$ of finite type, any unital $L$-full homomorphism $\iota: A\to D$, and any nuclear homomorphisms $\phi, \psi: A\to D$, if 
\begin{enumerate}
\item $[\phi]=[\psi]$ in $\mathrm{KL}(A, D)$,
\item $\phi(1_C)$ is unitarily equivalent to $\psi(1_C)$,
\end{enumerate}
then there exists a unitary $u\in \mathrm{M}_{n+1}(D)$ such that
$$\|u\mathrm{diag}(\phi(c), \underbrace{\iota(c), ..., \iota(c)}_n)u^* - \mathrm{diag}(\psi(c), \underbrace{\iota(c), ..., \iota(c)}_n) \| < \eps,\quad c\in\mathcal F.$$
\end{thm}

\begin{thm}\label{uniq-a-hom}
Let $A$ be a separable, unital, exact C*-algebra that satisfies the UCT. Let $L: A^1_+\setminus\{0\} \to \mathbb N$ be a map. 

Let $\mathcal F \subset A$ be a finite subset, and let $\eps>0$. Then there exist $n\in\mathbb N$, finite sets $\mathcal G, \mathcal H \subset A$, a constant $\delta >0$, and a finite set $\mathcal P\subset \mathrm{Proj}_\infty(A)$ such that, for any admissible codomain algebra $D$ of finite type, any unital $\mathcal G$-$\delta$-multiplicative maps $\phi, \psi, \iota: A\to D$, if 
\begin{enumerate}
\item $[\phi(p)]=[\psi(p)]$ for all $p\in\mathcal P$,
\item $\iota(a)$ is $L(a)$-full for any $a \in \mathcal H$,
\end{enumerate}
then there exists a unitary $u\in \mathrm{M}_{n+1}(D)$ such that
$$\|u\mathrm{diag}(\phi(c), \underbrace{\iota(c), ..., \iota(c)}_n)u^* - \mathrm{diag}(\psi(c), \underbrace{\iota(c), ..., \iota(c)}_n) \| < \eps,\quad c\in\mathcal F.$$
\end{thm}

\begin{proof}
Assume the statement were not true. There is a pair $(\mathcal F, \eps)$ such that for any $n$, any finite sets $\mathcal G, \mathcal H\subset A$, any constant $\delta$, any finite set $\mathcal P\subset\mathrm{Proj}_\infty(A)$, there is an admissible codomain algebra $D$ and unital $\mathcal G$-$\delta$-multiplicative maps satisfying 
\begin{enumerate}
\item $[\phi(p)]=[\psi(p)]$ for all $p\in\mathcal P$,
\item $\iota(a)$ is $L$-full for any $a \in \mathcal H$,
\end{enumerate}
but
$$\max_{c\in\mathcal F}\|u(\phi(c)\oplus(\bigoplus^n \iota(c)))u^* - \psi(c)\oplus(\bigoplus^n \iota(c)) \| \geq \eps$$
for any unitary $u\in \mathrm{M}_{n+1}(D)$.

Applying Theorem \ref{uniq-hom-1} to $(\mathcal F, \eps)$, one obtains $n$. Fix $(\mathcal F, \eps)$ and $n$. Then choose sequences
\begin{enumerate}
\item $\mathcal G_1\subset \mathcal G_{2}\subset\cdots $ with $\overline{\bigcup_{k=1}^\infty \mathcal G_k} =A$,
\item $\mathcal H_1\subset \mathcal H_{2}\subset\cdots $ with $\overline{\bigcup_{k=1}^\infty \mathcal H_k} =A^+$,
\item $\delta_1>\delta_2>\cdots$ with $\lim_{k\to\infty}\delta_k=0$,
\item $\mathcal P_1 \subset \mathcal P_2 \subset \cdots$ with $\bigcup\mathcal P_k = \underline{\mathrm{K}}(A)$.
\end{enumerate}
There are admissible codomain algebras $D_k$ and unital $\mathcal G_k$-$\delta_k$-multiplicative maps $\phi_k, \psi_k, \iota_k: A\to D_k$ satisfying 
\begin{enumerate}
\item $[\phi_k(p)]=[\psi_k(p)]$ for all $p\in\mathcal P_k$,
\item $\iota_k(a)$ is $\Delta$-full for any $a \in \mathcal H_k$,
\end{enumerate}
but
\begin{equation}\label{assumption}
\max_{c\in\mathcal F}\|u\mathrm{diag}\{\phi_k(c), \underbrace{\iota_k(c), ..., \iota_k(c)}_n\}u^* - \mathrm{diag}\{\psi_k(c), \underbrace{\iota_k(c), ..., \iota_k(c)}_n\} \| \geq \eps
\end{equation}
for any unitary $u\in \mathrm{M}_{n+1}(D_k)$.

Consider the homomorphisms $\Phi, \Psi, \Sigma: A \to \prod D_k/\bigoplus D_k$ define by $\prod_{k=1}^\infty \phi_k$,  $\prod_{k=1}^\infty \psi_k$,   and $\prod_{k=1}^\infty \iota_k$, respectively. Then one has
$$[\Phi]=[\Psi]\quad\textrm{in $KL(A, \prod D_k/\bigoplus D_k)$}.$$
Moreover, the map $\Sigma$ is $L$-full. Indeed, let $a$ be a non-zero element in $\bigcup \mathcal H_k$ and assume that $a\in\mathcal H_{k_0}$. Since $\mathcal H_k\subset\mathcal H_{k+1}$, for each $\mathcal H_{m}$, $m=k_0, k_0+1, ...,$ there exist contractions $x_{m, 1}, x_{m, 2}, ..., x_{m, \Delta(a)}$ such that 
$$1_{D_m}=x^*_{m, 1}\iota_m(a)x_{m, 1}+x^*_{m, 2}\iota_m(a)x_{m, 2}+\cdots+x^*_{m, L(a)}\iota_m(a)x_{m, L(a)}.$$
Consider the contractions
$$x_i:=(\underbrace{0, ..., 0}_{k_0-1}, x_{k_0, j}, x_{k_0+1, j}, ... ),\quad j=1, 2, ..., L(a).$$
One has
$$1_{\prod D_k/\bigoplus D_k}=x^*_1\Sigma(a)x_1+x^*_2\Sigma(a)x_2+\cdots+x^*_{L(a)}\Sigma(a)x_{L(a)}.$$
Since $A^+=\overline{\bigcup \mathcal H_k}$, one has that $\Sigma$ is $L$-full. By Theorem \ref{uniq-hom-1}, there is a unitary $$U\in \mathrm{M}_{n+1}(\prod D_k/\bigoplus D_k)$$ such that 
$$\|U^*\mathrm{diag}\{\Phi(a),\underbrace{\Sigma(a), ..., \Sigma(a)}_n\}U - \mathrm{diag}\{\Psi(a), \underbrace{\Sigma(a), ..., \Sigma(a)}_n\}\|<\eps,\quad a\in\mathcal F.$$
Then, by lifting $U$ to a unitary in $\prod D_k$, one obtains a contradiction to \eqref{assumption}.
\end{proof}

\begin{cor}\label{uniq-Q-1}
Let $A$ be a unital separable exact C*-algebra satisfying the UCT. Let $\rho: A^+\to (0, \infty)$ be a map.

Let $\mathcal F\subset A$ be a finite set, and let $\eps>0$. Then there is $n\in\mathbb N$, finite sets $\mathcal P\subset\mathrm{Proj}_\infty(A)$, $\mathcal G, \mathcal H\subset A$ and $\delta>0$ such that if $\phi, \psi, \sigma: A \to Q$ are unital $\mathcal G$-$\delta$-multiplicative map satisfying 
\begin{enumerate}
\item $[\phi(p)]_0=[\psi(p)]_0,$ $p\in\mathcal P$,  and
\item $\mathrm{tr}(\sigma(a))> \rho(a)$, $a\in \mathcal H$,
\end{enumerate}
then there is a unitary $u\in\mathrm{M}_{n+1}(Q)$ such that
$$\|u^*\mathrm{diag}(\phi(a), \underbrace{\sigma(a), \sigma(a), ..., \sigma(a)}_n)u - \mathrm{diag}(\psi(a), \underbrace{\sigma(a), \sigma(a), ..., \sigma(a)}_n)\|<\eps, \quad a\in\mathcal F.$$
\end{cor}

\begin{proof}
By Lemma 12.4 of \cite{GLN-TAS},  there is a function $L: A^+\setminus\{0\} \to \mathbb N$ such that if $\mathrm{tr}(b)> \rho(a)$ for some $b\in Q^+$, then $b$ is $L(a)$-full. Then the statement follows directly from Theorem \ref{uniq-a-hom}.
\end{proof}

\begin{cor}\label{uniq-Q}
Let $A$ be a separable exact C*-algebra (with or without a unit) satisfying the UCT. Let $\rho: A^+\setminus\{0\}\to (0, \infty)$ be a map.

Let $\mathcal F\subset A$ be a finite set, and let $\eps>0$. Then there is $n\in\mathbb N$, finite sets $\mathcal P\subset\mathrm{Proj}_\infty(\widetilde{A})\cap\Kzero(A)$, $\mathcal G, \mathcal H\subset A$ and $\delta>0$ such that if $\phi, \psi, \sigma: A \to Q$ are $\mathcal G$-$\delta$-multiplicative map satisfying 
\begin{enumerate}
\item $[\phi(p)]_0=[\psi(p)]_0,$ $p\in\mathcal P$,  and
\item $\mathrm{tr}(\sigma(a))> \rho(a)$, $a\in \mathcal H$,
\end{enumerate}
then there is a unitary $u\in\mathrm{M}_{n+1}(Q)$ such that
$$\|u^*\mathrm{diag}(\phi(a), \underbrace{\sigma(a), \sigma(a), ..., \sigma(a)}_n)u - \mathrm{diag}(\psi(a), \underbrace{\sigma(a), \sigma(a), ..., \sigma(a)}_n)\|<\eps, \quad a\in\mathcal F.$$

\end{cor}

\begin{proof}
If $A$ is unital, then it follows from Corollary \ref{uniq-Q-1} directly. Then let us assume that $A$ is nonunital.

Consider the unitization $\widetilde{A}$. Since $A$ is nonunital, $A$ is an essential ideal of $\widetilde{A}$ (as if $(a, -\lambda)\perp A$, then a straightforward calculation shows that $a$ must be the unit of $A$ and $\lambda=1$), and therefore $a^{\frac{1}{2}}Aa^{\frac{1}{2}}\neq\{0\}$ for any nonzero positive element $a\in\widetilde{A}$. Then there is a nonzero positive element $a'\in A$ such that $a' \leq a$. So there is a map $J: \widetilde{A}^+\setminus\{0\}\to A^+\setminus\{0\}$ such that 
$$a\geq J(a),\quad a\in \widetilde{A}^+\setminus\{0\}.$$
Then put $$\widetilde{\rho}(a) = \rho(J(a)),\quad a\in \widetilde{A}^+\setminus\{0\}.$$

Applying Corollary \ref{uniq-Q-1} to $(\mathcal F, \eps)$ with respect to $\widetilde{A}$ and $\widetilde{\rho}$, one obtains $n$, $\widetilde{\mathcal P}$, $\widetilde{\mathcal G}, \widetilde{\mathcal H} \subset \widetilde{A}$, and $\delta>0$. Choose finite subset $\mathcal G \subset A$ such that if $\theta: A\to Q$ is $\mathcal G$-$\delta$-multiplicative, then the unitization $\widetilde{\theta}: \widetilde{A} \to Q$ is $\widetilde{\mathcal G}$-$\delta$-multiplicative; set $$\mathcal H=J(\widetilde{\mathcal H});$$ since $\Kzero(\widetilde{A}) = \Kzero(A)\oplus\Kzero(\Comp 1_{\widetilde{A}})$, one may assume that any element of $\widetilde{\mathcal P}$ has the form of $p - q$ with $\pi(p) = \pi(q)$, and thus set $\mathcal P = \widetilde{\mathcal P} \subset \mathrm{Proj}_\infty(\widetilde{A})\cap\Kzero(A)$.

It now follows that the tuple $(n, \mathcal P, \mathcal G, \mathcal H)$ satisfies the conclusion of the corollary.

Indeed, let $\phi, \psi, \sigma: A\to Q$ be maps satisfying the assumption of the corollary. Consider the unitizations $\widetilde{\phi}, \widetilde{\psi}, \widetilde{\sigma}: \widetilde{A}\to Q$. Then, for each $a\in\widetilde{\mathcal H}$, 
$$\mathrm{tr}(\widetilde{\sigma}(a)) \geq \mathrm{tr}(\widetilde{\sigma}(J(a))) =  \mathrm{tr}(\sigma(J(a))) \geq \rho(J(a)) =  \widetilde{\rho}(a).$$
Since $\mathcal P = \widetilde{\mathcal P}$, it is clear that $$[\widetilde{\phi}(p)]_0 = [\widetilde{\psi}(p)]_0,\quad p\in \widetilde{\mathcal P}.$$
Then it follows from Corollary \ref{uniq-Q-1} that there is a unitary $u\in\mathrm{M}_{n+1}(Q)$ such that
$$\|u^*\mathrm{diag}(\widetilde{\phi}(a), \underbrace{\widetilde{\sigma}(a), \widetilde{\sigma}(a), ..., \widetilde{\sigma}(a)}_n)u - \mathrm{diag}(\widetilde{\psi}(a), \underbrace{\widetilde{\sigma}(a), \widetilde{\sigma}(a), ..., \widetilde{\sigma}(a)}_n)\|<\eps, \quad a\in\mathcal F.$$
Since $\widetilde{\phi}, \widetilde{\psi}, \widetilde{\sigma}$ are extensions of $\phi, \psi, \sigma$ respectively, one has
$$\|u^*\mathrm{diag}({\phi}(a), \underbrace{{\sigma}(a), {\sigma}(a), ..., {\sigma}(a)}_n)u - \mathrm{diag}({\psi}(a), \underbrace{{\sigma}(a), {\sigma}(a), ..., {\sigma}(a)}_n)\|<\eps, \quad a\in\mathcal F,$$
as desired.
\end{proof}

\begin{cor}\label{uniq-Q-path}
Let $A$ be a separable exact C*-algebra (with or without unit) satisfying the UCT. Let $\rho: A^+\setminus\{0\}\to (0, \infty)$ be a map.

Let $\mathcal F\subset A$ be a finite set, and let $\eps>0$. Then there are finite sets $\mathcal P\subset\mathrm{Proj}_\infty(\widetilde{A})\cap\Kzero(A)$, $\mathcal G, \mathcal H\subset A$ and $\delta>0$ such that if $\phi_0, \phi_1: A \to Q$ are $\mathcal G$-$\delta$-multiplicative maps satisfying 
\begin{enumerate}
\item $[\phi_0(p)]_0=[\phi_1(p)]_0,$ $p\in\mathcal P$,  and
\item $\mathrm{tr}(\phi_0(a))> \rho(a)$ and $\mathrm{tr}(\phi_1(a))> \rho(a)$, $a\in \mathcal H$
\end{enumerate}
then there is a path $\{\phi_t$, $t\in[0, 1]\}$ connecting $\phi_0$ and $\phi_1$ such that
$\phi_t$ is $\mathcal F$-$\eps$-multiplicative for each $t\in [0, 1]$, and 
$$| \mathrm{tr}(\phi_0(a)-\phi_t(a)) | < | \mathrm{tr}(\phi_0(a)-\phi_1(a)) | + 2\eps,\quad a \in\mathcal F,\ t\in[0, 1].$$
\end{cor}

\begin{proof}
Applying Corollary \ref{uniq-Q} to $(\mathcal F, \eps)$ with respect to $L$,  one obtains $\mathcal P$, $\mathcal G$, $\mathcal H$, $\delta$, and $n$. Then $\mathcal P$, $\mathcal G$, $\mathcal H$ and $\delta$ satisfy the conclusion of the statement.

Indeed, let $\phi_0, \phi_1: A \to Q$ be $\mathcal G$-$\delta$-multiplicative maps satisfying 
\begin{enumerate}
\item $[\phi_0(p)]_0=[\phi_1(p)]_0,$ $p\in\mathcal P$,  and
\item $\mathrm{tr}(\phi_0(a))> \rho(a)$ and $\mathrm{tr}(\phi_1(a))> \rho(a)$, $a\in \mathcal H$.
\end{enumerate}

Since the map $a\mapsto a\otimes 1_{2n+1}: Q \to Q\otimes \mathrm{M}_{2n+1}\cong Q$ is approximately unitarily equivalent to the identity map,  without loss of generality, one may assume that $\phi_0=\phi'_0\otimes 1_{2n+1}$ and $\phi_1=\phi'_1\otimes 1_{2n+1}$ for $\mathcal G$-$\delta$-multiplicative maps $\phi'_0, \phi'_1: A\to Q$ satisfying
\begin{enumerate}
\item $[\phi_0(p)]_0=[\phi_1(p)]_0,$ $p\in\mathcal P$,  and
\item $\mathrm{tr}(\phi_0(a))> \rho(a)$ and $\mathrm{tr}(\phi_1(a))> \rho(a)$, $a\in \mathcal H$.
\end{enumerate}
Then, with Corollary \ref{uniq-Q}, a standard argument shows that the desired path exists.
\end{proof}

\section{Trace factorization for C*-algebras with trivial $\Kzero$-groups}

\begin{thm}\label{TFA}
Let $A\in\mathcal R$ be a simple separable stably finite nuclear C*-algebra satisfying the UCT. Assume that $A$ has stable rank one and that $\Kzero(A)=\{0\}$. Then, for any finite set $\mathcal F\subseteq A$ and any $\eps>0$, there exist a Razak algebra $S$, an $\mathcal F$-$\eps$-multiplicative map $\theta: A\to S$, and an embedding $\iota: S \to A$ such that
$$ | \tau(a) - \tau(\iota\circ\theta(a)) | < \eps,\quad a\in\mathcal F,\ \tau\in \mathrm{T}^+A, \ \|\tau\|=1.$$
\end{thm}
\begin{proof}
Let us first assume that $A\cong A\otimes Q$.

Since $A$ is simple, any tracial state of $A$ is faithful. Since the cone $\mathrm T^+A$ has a compact base, one has $$\inf\{\tau(a): \tau\in\mathrm{T}^+(A), \|\tau\|=1\}>0,\quad a\in A^+\setminus\{0\}.$$
Therefore, one can define the map
\begin{equation}\label{L-defn}
L: A^+\setminus\{0\} \ni a \mapsto \inf\{\tau(a): \tau\in\mathrm{T}^+A, \|\tau\| =1\}\in(0, \infty).
\end{equation}


Applying Corollary \ref{uniq-Q-path} to $A$ and $L/4$ with respect to $(\mathcal F, \eps)$, one obtains $\mathcal G\subset A$, $\mathcal H\subset A^+$, and $\delta>0$ satisfying the conditions of Corollary \ref{uniq-Q}. (Since $\Kzero(A)=\{0\}$ and $A$ is stably finite, one has $\mathrm{Proj}_\infty(A)\cap\Kzero(A)=\{0\}$, and hence $\mathcal P=\{0\}$.)

Choose a model algebra $S=\varinjlim(S_n, \iota_n)$ such that each $S_n$ is a $Q$-stablized Razak algebra, and  
$$\Kzero(S)=\{0\}\quad\mathrm{and}\quad (\mathrm{T}^+S, \|\cdot\|)\cong (\mathrm{T}^+A, \|\cdot\|).$$
Denote by $\gamma: \aff\mathrm{T}A \to \aff\mathrm{T}S$ the induced isomorphism.

By Lemma \ref{trace-lifting}, there is positive continuous affine map $\gamma_n: \aff\mathrm{T}^+A \to \aff\mathrm{T}^+S_n$ for a sufficiently large $n$ such that 
\begin{equation}\label{tr-appr-1}
| \gamma_n^*\circ\iota^*_{n, \infty}(\tau)(a) -  \gamma^*(\tau)(a) | <\eps,\quad a\in\mathcal F\cup\mathcal H,\ \tau\in \mathrm{T}^+A,\ \|\tau\|=1
\end{equation}
and 
\begin{equation}\label{tr-appr-2}
1-\eps'<\|\gamma_n^*(\tau)\|<1+\eps,\quad \tau\in\mathrm{T}^+S_n, \ \|\tau\|=1.
\end{equation}

Write $$S_n=\{(e, f) \in E\oplus\mathrm{C}([0, 1], F): f(0)=\phi_0(e),\ f(1)=\phi_1(e)\},$$
where $E=\bigoplus_{i=1}^p Q$ and $F=\bigoplus_{i=1}^l Q$ for some $p, l\in\mathbb N$, and $\phi_0, \phi_1: E\to F$ are homomorphisms. The maps $\phi_0$ and $\phi_1$ might not be unital, but by Condition \ref{model-cond-2}, one may assume that 
\begin{equation}\label{am-state}
\tau(\phi_0(1_E))>\frac{1}{2}\quad\mathrm{and}\quad \tau(\phi_1(1_E))>\frac{1}{2},\quad \tau\in\mathrm{T}_1(F).
\end{equation}

For $i$th direct summand $Q$ of $E$, consider its canonical tracial state $\mathrm{tr}_i$ and consider the trace $$\gamma_n^*(\tr\circ\pi_{\infty, i}) \in \mathrm{T}(A).$$
It follows from \eqref{tr-appr-2} that 
$$1-\eps< \|\gamma_n^*(\tr\circ\pi_{\infty, i})\| < 1+\eps.$$

Take a partition $$0=t_0< t_1 < t_2 < \cdots < t_s <t_{s+1}=1$$ such that
$$| \gamma_n(a)(\mathrm{tr}_{j, t_k}) - \gamma_n(a)(\mathrm{tr}_{j, t}) | < \eps,\quad a\in\mathcal F,\ t\in[t_k, t_{k+1}],\  1\leq j\leq l,\ 0 \leq k\leq s.$$ 

Since $A$ has the UCT, by \cite{TWW-QD}, there are $\mathcal G$-$\delta$-multiplicative maps $\psi_{\infty, i}: A\to Q$, $i=1, 2, ..., p$, such that
\begin{equation}\label{qtr-0}
|\mathrm{tr}\circ\psi_{\infty, i}(a) - (\gamma_n)_*(\tr_i\circ\pi_{\infty, i})(a)|<\frac{1}{2}\min\{\eps, L(h): h\in\mathcal H\},\quad a\in\mathcal F\cup \mathcal H,
\end{equation}
and there are $\mathcal G$-$\delta$-multiplicative maps $\psi_{j, t_k}: A\to Q$, $j=1, 2, ..., l$, $k=1, 2, ..., s$, such that 
\begin{equation}\label{qtr-1}
|\mathrm{tr}\circ\psi_{j,t_ k}(a) - (\gamma_n)_*(\tr\circ\pi_{j, t_k})(a)|<\frac{1}{2}\min\{\eps, L(h): h\in\mathcal H\},\quad a\in\mathcal F\cup \mathcal H.
\end{equation}
Set $$\psi_\infty:=\bigoplus_{i=1}^p\psi_{\infty, i}: A\to E.$$

Note that, by \eqref{qtr-0}, \eqref{qtr-1}, \eqref{tr-appr-2}, and the construction of $L$ (\eqref{L-defn}), one has
\begin{equation}\label{qtr-2}
\mathrm{tr}(\psi_{\infty, i})(a) >\frac{1}{2}L(a)\quad\mathrm{and}\quad \mathrm{tr}(\psi_{j, t_k})(a) >\frac{1}{2}L(a),\quad a\in\mathcal H,
\end{equation}
where $i=1, 2, ..., p$, $j=1, 2, ..., l$, and $k=1, 2, ..., s$.

For each $j=1, 2, ..., l$, consider the maps 
$$\psi_{j, t_0}=\pi_j\circ\phi_0\circ \psi_{\infty} \quad\mathrm{and}\quad \psi_{j, t_{s+1}}=\pi_j\circ\phi_1\circ \psi_{\infty}.$$
Note that, by \eqref{am-state} and \eqref{qtr-2}, one has
$$\mathrm{tr}(\psi_{j, t_0})(a)>\frac{1}{4}L(a) \quad\mathrm{and}\quad \mathrm{tr}(\psi_{j, t_{s+1}})(a)>\frac{1}{4}L(a),\quad a\in\mathcal H.$$

Compare $\psi_{j, t_{k}}$ with $\psi_{j, t_{k+1}}$ for each $k=0, 1, ..., s$. By Corollary \ref{uniq-Q-path}, exists a path $(\psi_{j, t_{k}})_t$, $t\in [t_k, t_{k+1}]$ such that $(\psi_{j, t_{k}})_t: A \to Q$ is $\mathcal F$-$\eps$-multiplicative for each $t\in [t_k, t_{k+1}]$, and
$$| \mathrm{tr}((\psi_{j, t_{k}})_t(a) - \psi_{j, t_{k}}(a)) | <3 \eps,\quad a\in\mathcal F,\ t\in [t_k, t_{k+1}].$$

Denote by $\theta: A \to S_n$ the induced map. It is $\mathcal F$-$\eps$-multiplicative and satisfies
\begin{equation}\label{tr-est-AC}
 | \tau(\Psi(a))- \gamma_n(a)(\tau) | < 4\eps,\quad a\in\mathcal F,\quad \tau\in \mathrm{T}^+S_n,\ \|\tau\| = 1. 
 \end{equation}

Note that $A$ and $S$ are stably projectionless and $\mathrm{T}^+A \cong \mathrm{T}^+S$. By Theorem 4.4 and Corollary 6.8 of \cite{ESR-Cuntz}, it follows that the Cuntz semigroup of $A$ and the Cuntz semigroup of $S$ are isomorphic. Applied to the canonical map $\mathrm{Cu}(S_n)\to\mathrm{Cu}(S)\cong\mathrm{Cu}(A)$, Theorem 1 of \cite{Robert-Cu} implies that there is a homomorphism $\iota: S_n\to A$ giving rise to this map, and in particular such that 
\begin{equation}\label{tr-est-CA}
\iota_*=\gamma^{-1}\circ(\iota_{n, \infty})_*\quad\textrm{on $\aff\mathrm{T}^+S_n$}.
\end{equation}
Since the ideal of $\mathrm{Cu}(S_n)$ killed by the map $\mathrm{Cu}(S_n)\to\mathrm{Cu}(S)\cong\mathrm{Cu}(A)$ is zero, as the map $S_n\to S$ is an embedding, it follows that the map $S_n\to A$ is also an embedding. 
By \eqref{tr-est-AC}, \eqref{tr-est-CA}, and \eqref{tr-appr-1}, one then has 
$$\|\iota_*\circ\theta_*(\hat{f})-\hat{f}\|_\infty<6\eps,\quad f\in\mathcal F,$$
and this proves the theorem in the case that $A\cong A\otimes Q$.

In the general case, consider the embedding $A\to A\otimes Q$, $a\mapsto a\otimes 1_A$, and note that it induced an isometric isomorphism from $\mathrm{T}^+A$ to $\mathrm{T}^+(A\otimes Q)$. Then, for any $\mathcal F\subset A$, and any $\eps$, the prove above provides a sub-C*-algebras $S$ and a factorization
\begin{displaymath}
\xymatrix{
A \ar[r] & A\otimes Q \ar[r]^-\theta & S \ar[r]^-\iota & A\otimes Q
}
\end{displaymath}
such that
$$ | \tau(a) - \tau(\iota\circ\theta(a\otimes 1_Q)) | < \eps,\quad a\in\mathcal F,\ \tau\in \mathrm{T}^+(A\otimes Q), \ \|\tau\|=1.$$

Since $A$ and $A\otimes Q$ are stably projectionless, one has that $\mathrm{Cu}(A\otimes Q) \cong \mathrm{Cu}(A)$, and denote this isomorphism by $\xi$. Then, by Theorem 1 of \cite{Robert-Cu}, there is an embedding $\iota': S\to A$ such that $[\iota'] = \xi\circ[\iota]$ on $\mathrm{Cu}(S)$, and hence
$$ | \tau(a) - \tau(\iota'\circ\theta(a\otimes 1_Q)) | < \eps,\quad a\in\mathcal F,\ \tau\in \mathrm{T}^+(A), \ \|\tau\|=1,$$
as desired.
\end{proof}

\section{Tracial Approximations}

\begin{lem}[\cite{CE-str1} or Theorem 6 of \cite{Robert-Cu}]\label{pre-comp}
Consider $I=\mathrm{C}_0((0, 1])$. For any finite set $\mathcal F\subset I^+$ and any $\eps>0$, there are finite set $\mathcal G\subset \mathrm{M}_\infty(I)$ such that for any C*-algebra $A$ with stable rank one and any homomorphisms $\phi, \psi: I\to A$, if
$$[\phi(g')] \leq [\psi(g)]\quad\mathrm{and}\quad [\psi(g')] \leq \textrm{$[\phi(g)],\quad g', g\in\mathcal G$ with $g'\ll g$},$$
then there is a unitary $u\in \tilde{A}$ such that
$$\|u^*\phi(a) u -\psi(a)\|<\eps,\quad a\in \mathcal F.$$
\end{lem}

\begin{lem}\label{comp}
Let $A$ be a simple C*-algebra (with or without unit) with stable rank one such that $A\in\mathcal R$, and assume that $A$ has strict comparison. 

Let $F$ be a finite dimensional C*-algebra, and let
$$\phi: F\to A,$$
$$\phi_i: F\to A,\quad i\in \mathbb N,$$
be order zero maps such that for each $c\in F_+$ and $f\in\mathrm{C}_{0}^{+}((0, 1])$,
$$\lim_{i\to\infty}\sup_{\tau\in\Sigma A} |\tau(f(\phi(c))-f(\phi_i(c)))| = 0$$
and
$$\limsup_{i\to\infty} \|f(\phi_i(c))\| \leq \|f(\phi(c))\|.$$
Then, there are contractions
$$s_i\in \mathrm{M}_4 \otimes  A,\quad i\in\mathbb N,$$
such that $$\lim_{i\to\infty}\|s_i(1_4\otimes \phi(c)) - (e_{1, 1} \otimes \phi_i(c))s_i\| = 0,\quad c\in F^+$$
and $$\lim_{i\to\infty}\|(e_{1, 1} \otimes \phi_i(c))s_is_i^*- e_{1, 1} \otimes \phi_i(c)\| = 0. $$
\end{lem}

\begin{proof}
The proof is the same as Proposition 2.1 of \cite{Winter-TA} (the argument does not require the C*-algebra to be unital; the strict comparison is sufficient for the argument to proceed).
\end{proof}

\begin{lem}\label{app-unit}
Let $A$ be a separable C*-algebra with nuclear dimension at most $m$. Let $(e_n)$ be an approximate unit for $A$.
Then there is a system of $(m+1)$-decomposable c.p.~approximations
$$
\xymatrix{
\tilde{A} \ar[r]^-{\tilde{\psi}_j} & F_j^{(0)}\oplus F_j^{(1)}\oplus\cdots\oplus F_j^{(m)} \oplus \Comp \ar[r]^-{\tilde{\phi}_j} & \tilde{A},\quad j=1, 2, ...
}
$$ 
such that 
\begin{equation}\label{m-decom-cond-1}
\tilde{\phi}_j(F_j^{(l)})\subset A,\quad l=0, 1, ..., m,
\end{equation}
\begin{equation}\label{m-decom-cond-2}
\tilde{\phi}_j|_\Comp(1_{\Comp}) = 1_{\tilde{A}}-e_{n_j},\quad\textrm{for some $e_{n_j}$ in the approximate unit $(e_n)$}, 
\end{equation} 
and
\begin{equation}\label{refine-0} 
\lim_{j\to\infty}\|a\tilde{\phi}^{(l)}_j\tilde{\psi}^{(l)}_j(1_{\tilde{A}}) - \tilde{\phi}_j^{(l)}\tilde{\psi}^{(l)}_j(a)\|=0,\quad l=1, 2, ..., m+1, \ a\in \tilde{A},
\end{equation} 
where $\tilde{\phi}^{(l)}_j$ and $\tilde{\psi}^{(l)}_j$ are the projection of $\tilde{\phi}_j$ to $F_j^{(l)}$ and the restriction of $\tilde{\psi}_j$ to $F_j^{(l)}$, respectively.
\end{lem}

\begin{proof}
Let $\mathcal F\subseteq \tilde{A}$ be a finite set of positive elements with norm one,  and let $\eps>0$ be arbitrary. 
Write $$\mathcal F=\{(a, \lambda_a): a\in \mathcal F_A\subset A^+\}\subset A\oplus \Comp 1_{\tilde{A}}.$$
By considering a strictly positive element of $A$, there is $e$ in the approximate unit $(e_n)$ of $A$ such that for any $a\in \mathcal F$, there is $a'\in A$ satisfying
$$\|a-a'\|<\eps,\quad a'e=ea',\quad\mathrm{and}\quad (a'-\pi(a'))(1-e)=0,$$
where $\pi: \tilde{A} \to \Comp$ is the canonical quotient map.
Denote by $\mathcal F'$ the set of such $a'$.

Consider the elements $$e^{\frac{1}{2}} a' e^{\frac{1}{2}} \in A,\quad a'\in\mathcal F'.$$ Then choose a factorization
$$
\xymatrix{
A \ar[r]^-{\psi} & F^{(0)}\oplus F^{(1)}\oplus\cdots\oplus F^{(m)} \ar[r]^-{\phi} & A
}
$$ 
such that 
\begin{enumerate}
\item $\| \phi(\psi(e^{\frac{1}{2}}a'e^{\frac{1}{2}})) - e^{\frac{1}{2}}a'e^{\frac{1}{2}} \| <\eps$, $a'\in\mathcal F',$ and
\item the restriction of $\phi$ to each direct summand $F^{(l)}$, $l=0, 1, ..., m$, is order zero.
\end{enumerate}

Then, define maps
$$\tilde{\psi}: \tilde{A}\ni a \mapsto \psi(e^{\frac{1}{2}}ae^{\frac{1}{2}})\oplus \pi(a)  \in (F^{(0)}\oplus F^{(1)}\oplus\cdots\oplus F^{(m)})\oplus\Comp, $$
and 
$$\tilde{\phi}: (F^{(0)}\oplus F^{(1)}\oplus\cdots\oplus F^{(m)})\oplus\Comp\ni (a, \lambda) \mapsto \phi(a) + \lambda(1-e). $$

Then, for any $a\in\mathcal F$, one has
\begin{eqnarray*}
\|\tilde{\phi}(\tilde{\psi}(a)) - a\| & < & \|\tilde{\phi}(\tilde{\psi}(a')) - a'\|+ 3\eps \\
& = &\|\phi(\psi(e^{\frac{1}{2}}a'e^{\frac{1}{2}})) +\pi(a')(1-e) -a' \| + 3\eps \\
& < & \| e^{\frac{1}{2}}a'e^{\frac{1}{2}} + \pi(a')(1-e) - a'\| + 4\eps=4\eps.
\end{eqnarray*}
It is clear that the restriction of $\tilde{\phi}$ to each simple summand of $F^{(0)}\oplus F^{(1)}\oplus\cdots\oplus F^{(m)}\oplus\Comp$ has order zero. 

Since $\mathcal F$ and $\eps$ are arbitrary, one obtains the $(m+1)$-decomposable c.p.~approximations $(\tilde{\psi}_j, \tilde{\phi}_j)$, $j=1, 2, ...$, which satisfy \eqref{m-decom-cond-1} and \eqref{m-decom-cond-2} of the Lemma.


In the same way as in the proof of Proposition 4.2 of \cite{Winter-Z-stable-02}, $\tilde{\psi}_j$ and $\tilde{\phi}_j$ can be modified to satisfy \eqref{refine-0}. Indeed, consider
$$\hat{\psi}_j: \tilde{A} \ni a \mapsto \tilde{\psi}_j(1_{\tilde{A}})^{-\frac{1}{2}}\tilde{\psi}_j(a) \tilde{\psi}_j(1_{\tilde{A}})^{-\frac{1}{2}} \in (F_j^{(0)}\oplus F_j^{(1)}\oplus\cdots\oplus F_j^{(m)})\oplus\Comp,$$
where the inverse is taking in the hereditary sub-C*-algebra generated by $\tilde{\psi}(1_{\tilde{A}})$,
and
$$\hat{\phi}_j: (F_j^{(0)}\oplus F_j^{(1)}\oplus\cdots\oplus F_j^{(m)})\oplus\Comp \ni a\mapsto \phi(\tilde{\psi}_j(1_{\tilde{A}})^{\frac{1}{2}} a \tilde{\psi}_j(1_{\tilde{A}})^{\frac{1}{2}}) \in \tilde{A}.$$
Then the proof of Proposition 4.2 of \cite{Winter-Z-stable-02} shows that
$$\lim_{j\to\infty}\|\hat{\phi}_j^{(l)}\hat{\psi}^{(l)}_j(a) - \hat{\phi}^{(l)}_j\hat{\psi}^{(l)}_j(a) \hat{\phi}^{(l)}_j\hat{\psi}^{(l)}_j(1_{\tilde{A}}) \|=0,\quad a\in A.$$

Note that $\pi(1_{\tilde{A}}) = 1_{\Comp}$. One has that $\pi(1_{\tilde{A}})\Comp\pi(1_{\tilde{A}})=\Comp$, and the restriction of $\hat{\phi}$ to $\Comp$ is the map $\lambda \mapsto \lambda(1-e)$. It follows that the decompositions $(\hat{\psi}_j, \hat{\phi}_j)$ satisfies the lemma.
\end{proof}
%
 
\begin{thm}\label{fdim}
Let $\mathcal S$ be a class of nuclear C*-algebra.

Let $A$ be a C*-algebra with nuclear dimension at most $m$ such that $A\in\mathcal R$, $A$ has stable rank one, and $A$ has strict comparison of positive elements. 

Assume that for any finite set $\mathcal F\subset A$ and any $\eps>0$, there is a C*-algebra $S\in\mathcal S$ and maps
\begin{displaymath}
\xymatrix{
A \ar[r]^\sigma & S \ar[r]^\rho & A
}
\end{displaymath}
such that $\sigma$ is $\mathcal F$-$\eps$-multiplicative, $\rho$ is an embedding, such that
$$| \tau(\rho(\sigma(a))) - \tau(a)|<\eps, \quad a\in\mathcal F,\  \tau\in \mathrm{T}^+A,\ \|\tau\|=1.$$

Then, the C*-algebra $A$ has the following property: For any finite set $\mathcal F\subset A$ and any $\eps>0$, there is a projection $p\in \mathrm{M}_{4(m+2)}(\tilde{A})$, a sub-C*-algebra $S\subset (1-p)\mathrm{M}_{4(m+2)}(A)(1-p)$ with $S\in\mathcal S$, such that
\begin{enumerate}
\item $\|[p, 1_{4(m+2)} \otimes a]\| < \eps$, $a\in\mathcal F$,
\item $p(1_{4(m+2)} \otimes a) p\in_\eps S$, $a\in\mathcal F$,
\item $p\sim e_{11}$ in $\mathrm{M}_{4(m+2)}(\tilde{A})$, and
\item $(1-p)\mathrm{M}_{4(m+2)}(A)(1-p) \in \mathcal R$.
\end{enumerate}
\end{thm}

\begin{proof}
The proof is similar to that of Theorem 2.2 of \cite{Winter-TA}. Let $(e_n)$ be an approximate unit of $A$. Since $A\in\mathcal R$, by Lemma \ref{full-sp} one may assume that $\mathrm{sp}(e_n)=[0, 1]$. Since $\mathrm{dim}_\mathrm{nuc}(A)\leq m$, by Lemma \ref{app-unit}, there is a system of $(m+1)$-decomposable c.p.~approximations
$$
\xymatrix{
\tilde{A} \ar[r]^-{\psi_j} & F_j^{(0)}\oplus F_j^{(1)}\oplus\cdots\oplus F_j^{(m)} \oplus \Comp \ar[r]^-{\phi_j} & \tilde{A},\quad j=1, 2, ...
}
$$ 
such that 
\begin{equation}
\phi_j(F_j^{(l)})\subset A, \quad l=0, 1, ..., m
\end{equation} and 
\begin{equation}\label{defn-phi-j}
\phi_j|_\Comp(1_\Comp) = 1_{\tilde{A}}-e_{j},
\end{equation} 
where $e_j$ is an element of $(e_n)$.

Denote by
$$\phi_{j}^{(l)} = \phi_j|_{F_j^{(l)}}\quad\mathrm{and}\quad \phi_{j}^{(m+1)}(\lambda)= \phi_j|_\Comp,\quad l=0, 1, ..., m.$$
By Lemma \ref{app-unit}, one may assume that
\begin{equation}\label{refine}
\lim_{j\to\infty}\|a\phi^{(l)}_j\psi^{(l)}_j(1_{\tilde{A}}) - \phi_j^{(l)}\psi^{(l)}_j(a)\|=0,\quad l=1, 2, ..., m+1, \ a\in A.
\end{equation}

Then, by the weak stability of order zero maps (Proposition 1.4 of \cite{WZ-ndim}), for each $j\in\mathbb N$, there are order zero maps
$$\tilde{\phi}_{j, i}^{(l)}: F_j^{(l)} \to S_i\subset A,\quad i\in\mathbb N,\ l=0, 1, ..., m,$$
such that
$$\lim_{i\to\infty}\| \tilde{\phi}_{j, i}^{(l)}(c) - \sigma_i(\phi_{j}^{(l)}(c))\|=0,\quad c\in F_{j}^{(l)};$$
and
(order zero) positive linear map 
$$\tilde{\phi}_{j, i}^{(m+1)}: \Comp\ni 1 \mapsto 1_{\tilde{A}}-\sigma_i(e_j)\in \tilde{S_i}=\textrm{C*}\{S_i, 1_{\tilde{A}}\}\subset \tilde{A},\quad i\in\mathbb N.$$
Note that
\begin{equation}\label{defn-phi-j-p}
 \tilde{\phi}_{j, i}^{(m+1)}(\lambda) = \sigma_i(\phi_{j}^{(m+1)}(\lambda)),\quad \lambda \in F_{j}^{(m+1)}=\Comp,
\end{equation} 
where one still uses $\sigma_i$ to denote the induced map $\tilde{A} \to \tilde{S_i}$

Note that for each $l=0, 1, ..., m$,
$$\lim_{i\to\infty}\|f(\tilde{\phi}_{j, i}^{(l)}(c)) - \sigma_i (f(\phi_{j}^{(l)}(c)))\| = 0,\quad c\in (F_j^{(l)})^+, \  f\in\mathrm{C}^+_0((0, 1]),$$
and hence
$$\lim_{i\to\infty}\sup_{\tau\in\Sigma A}| \tau(f(\tilde{\phi}_{j, i}^{(l)}(c)) - f(\phi_{j}^{(l)}(c)))| = 0,\quad c\in (F_j^{(l)})^+, \  f\in\mathrm{C}^+_0((0, 1]).$$
Also note that
$$\limsup_{i\to\infty} \|f(\tilde{\phi}_{j, i}^{(l)})(c)\| \leq \| f(\phi_j^{(l)})(c)\|, \quad c\in (F_j^{(l)})^+, \  f\in\mathrm{C}^+_0((0, 1]).$$

Applying Lemma \ref{comp} to $(\tilde{\phi}_{j, i}^{(l)})_{i\in\mathbb N}$ and $\phi_j^{(l)}$ for each $l=0, 1, ..., m$, there are contractions $$s^{(l)}_{j, i} \in \mathrm{M}_4(A) \subset \mathrm{M}_4(\tilde{A}),\quad i\in\mathbb N,$$
such that
\begin{equation*}
\lim_{i\to\infty}\|s_{j, i}^{(l)}(1_4\otimes \phi_j^{(l)}(c)) - (e_{1, 1}\otimes\tilde{\phi}_{j, i}^{(l)}(c)) s_{j, i}^{(l)}\| = 0,\quad c\in F_j^{(l)}
\end{equation*}
and $$\lim_{i\to\infty}\|(e_{1, 1} \otimes \phi_j^{(l)}(c))s_{j,i}^{(l)}(s_{j, i}^{(l)})^*- e_{1, 1} \otimes \phi_j^{(l)}(c)\| = 0. $$
For $l=m+1$, since $\mathrm{sp}(e_j) = [0, 1]$, by Lemma \ref{pre-comp}, there are unitaries $$s^{(m+1)}_{j, i} \in \tilde{A},\quad i\in\mathbb N,$$ such that $$\lim_{i\to\infty}\|s_{j, i}^{(m+1)}e_j-\sigma_i(e_j)s_{j, i}^{(m+1)}\|=0,$$ and hence
$$\lim_{i\to\infty}\|s_{j, i}^{(m+1)}(1_{\tilde{A}}-e_j)-(1_{\tilde{A}}-\sigma_i(e_j))s_{j, i}^{(m+1)}\|=0.$$
By \eqref{defn-phi-j} and \eqref{defn-phi-j-p}, one has
\begin{equation*}
\lim_{i\to\infty}\|s_{j, i}^{(m+1)}\phi_j^{(m+1)}(c) - \tilde{\phi}_{j, i}^{(m+1)}(c) s_{j, i}^{(m+1)}\| = 0,\quad c\in F_j^{(m+1)}=\Comp.
\end{equation*}

Consider $e_{1, 1}\otimes s_{j, i}^{(m+1)} \in \mathrm{M_4}\otimes \tilde{A}$, and still denote it by $s_{j, i}^{(m+1)}$, one has
\begin{equation*}
\lim_{i\to\infty}\|s_{j, i}^{(m+1)}(1_4\otimes \phi_j^{(m+1)}(c)) - (e_{1, 1}\otimes\tilde{\phi}_{j, i}^{(m+1)}(c)) s_{j, i}^{(m+1)}\| = 0,\quad c\in F_j^{(l)}
\end{equation*}
and $$(e_{1, 1} \otimes \phi_j^{(m+1)}(c))s_{j,i}^{(m+1)}(s_{j, i}^{(m+1)})^* = e_{1, 1} \otimes \phi^{(m+1)}_j(c). $$

Therefore,
\begin{equation}\label{pre-s-1}
\lim_{i\to\infty}\|s_{j, i}^{(l)}(1_4\otimes\phi_j^{(l)}(c)) - (e_{1, 1} \otimes \tilde{\phi}_{j, i}^{(l)}(c)) s_{j, i}^{(l)}\| = 0,\quad c\in F_j^{(l)},\ l=0, 1, ..., m+1.
\end{equation}

\begin{equation}\label{pre-s-1-1}
\lim_{i\to\infty}\|(e_{1, 1}\otimes \phi_i(c))s_{j,i}^{(l)}(s_{j, i}^{(l)})^*- \phi_i(c)\| = 0, \quad c\in F_j^{(l)},\ l=0, 1, ..., m+1.
\end{equation}


%

Let $\tilde{\sigma}_i: \tilde{A} \to \tilde{S_i}$ and $\tilde{\rho}_i: \tilde{S_i} \to\tilde{A}$ denote the unital maps induced by 
$\sigma_i: {A} \to {S_i}$ and $\rho_i: {S_i} \to {A}$, respectively.

Consider the contractions
$$s_j^{(l)} := (s_{j, i}^{(l)})_{i\in\mathbb N}\in (\mathrm{M}_4\otimes \tilde{A})_\infty,\quad l=0, 1, ..., m+1.$$ By \eqref{pre-s-1} and \eqref{pre-s-1-1}, they satisfy
$$s_j^{(l)}(1_4\otimes \bar\iota(\phi_j^{(l)}(c))) = (e_{1, 1}\otimes \bar\rho\bar\sigma({\phi}_j^{(l)}(c))) s_{j}^{(l)}$$
and
$$(e_{1,1}\otimes \bar{\rho}(\phi_j^{(l)}(c)))s_j^{(l)}(s_j^{(l)})^* = (e_{1,1}\otimes \bar{\rho}(\phi_j^{(l)}(c))),$$
where 
$$\bar\rho: \prod \tilde{S_i}/\bigoplus \tilde{S_i} \to \tilde{A}_\infty$$
is the homomorphism induced by $\tilde{\rho}_i$,  $\iota: \tilde{A}\to (\tilde{A})_\infty$ is the canonical embedding, $\tilde{S_i}=\mathrm{C^*}\{S_i, 1_{\tilde{A}}\}$, and
$$\bar\iota: (\tilde{A})_\infty \to ((\tilde{A})_\infty)_\infty$$ is the embedding induced by the canonical embedding $\iota: \tilde{A}\to (\tilde{A})_\infty$.

Let $$\bar\gamma: \tilde{A}_\infty \to ((\tilde{A})_\infty)_\infty$$ be the homomorphism induced by
$$\bar{\rho}\bar{\sigma}: \tilde{A} \to (\tilde{A})_\infty,$$
For each $l=0, 1, ..., m+1$, let 
$$\bar{\phi}^{(l)}: \prod_j F_j^{(l)}/\bigoplus_j F_j^{(l)} \to A_\infty$$
and
$$\bar{\psi}^{(l)}: A \to \prod_j F_j^{(l)}/\bigoplus_j F_j^{(l)}$$
denote the maps induced by $\phi_j^{(l)}$ and $\psi_j^{(l)}$. 

Define a contraction
$$\bar{s}^{(l)} = (s_j^{(l)}) \in (\mathrm{M}_4\otimes \tilde{A}_\infty)_\infty,$$ and then
$$s^{(l)}(1_4\otimes \bar\iota\bar{\phi}^{(l)}\bar\psi^{(l)}(a)) = (e_{1, 1} \otimes \bar\gamma\bar\phi^{(l)}\bar\psi^{(l)}(a)) s^{(l)},\quad a\in \tilde{A},$$ and
$$(e_{1, 1} \otimes \bar\gamma\bar\phi^{(l)}\bar\psi^{(l)}(a)) s_j^{(l)}(s_j^{(l)})^* = (e_{1, 1} \otimes \bar\gamma\bar\phi^{(l)}\bar\psi^{(l)}(a)).$$

By \eqref{refine}, one has
$$\bar\phi^{(l)}\bar\psi^{(l)}(1_A)\iota(a) = \bar\phi^{(l)}\bar\psi^{(l)}(a),\quad a\in A.$$
In particular, 
$$((\bar\phi^{(l)}\bar\psi^{(l)}(1_A))^{\frac{1}{2}}\iota(a) \in \textrm{C*}\{\bar\phi^{(l)}\bar\psi^{(l)}(A)\},$$
and hence
\begin{eqnarray}\label{twrist}
\bar{s}^{(l)}(1_4 \otimes (\bar{\iota}\bar{\phi}^{(l)}\bar{\psi}^{(l)}(1_{\tilde{A}}))^{\frac{1}{2}})(1_4\otimes \bar{\iota}\iota(a)) &= & \bar{s}^{(l)}(1_4\otimes \bar{\iota}(\bar{\phi}^{(l)}\bar{\psi}^{(l)}(1_{\tilde{A}})^{\frac{1}{2}}\iota(a)) \nonumber \\
& = & (e_{1, 1} \otimes \bar{\gamma}(\bar{\phi}^{(l)}\bar{\psi}^{(l)}(1_{\tilde{A}}))^{\frac{1}{2}}\iota(a)) \bar{s}^{(l)} \nonumber \\
& = & (e_{1, 1} \otimes \bar{\gamma}(\iota(a)\bar{\phi}^{(l)}\bar{\psi}^{(l)}(1_{\tilde{A}}))^{\frac{1}{2}}) \bar{s}^{(l)} \nonumber \\
& = & (e_{1, 1} \otimes \bar{\gamma}(\iota(a)))(e_{1, 1}\otimes \bar{\gamma}(\bar{\phi}^{(l)}\bar{\psi}^{(l)}(1_{\tilde{A}}))^{\frac{1}{2}})\bar{s}^{(l)}.
\end{eqnarray}

Set
\begin{eqnarray*}
\bar{v} & = & \sum_{l=0}^{m+1} e_{1, l} \otimes ((e_{1, 1} \otimes \bar{\gamma}\bar{\phi}^{l}\bar{\psi}^{(l)}(1_{\tilde{A}}))^{\frac{1}{2}}\bar{s}^{(l)}) \\
& = & \sum_{l=0}^{m+1} e_{1, l} \otimes (\bar{s}^{(l)}(1_4\otimes \bar{\iota}\bar{\phi}^{l}\bar{\psi}^{(l)}(1_{\tilde{A}}))^{\frac{1}{2}}) \in \mathrm{M}_{m+2}(\Comp) \otimes \mathrm{M}_4(\Comp)\otimes (A_\infty)_\infty.
\end{eqnarray*}

Then
\begin{eqnarray*}
\bar{v}\bar{v}^* & = & \sum_{l=0}^{m+1} e_{1, 1} \otimes (e_{1, 1} \otimes (\bar{\gamma}\bar{\phi}^{l}\bar{\psi}^{(l)}(1_{\tilde{A}}))) =  e_{1, 1} \otimes e_{1, 1} \otimes \bar{\gamma}(1_{\tilde{A}}).
\end{eqnarray*}
Thus, $\bar{v}$ is an partial isometry. 

Moreover, for any $a\in \tilde{A}$,
\begin{eqnarray*}
\bar{v}(1_{(m+2)}\otimes 1_4 \otimes\bar{\iota}\iota(a)) & = & \sum_{l=0}^{m+1} e_{1, l} \otimes (\bar{s}^{(l)}(1_4\otimes \bar{\iota}\bar{\phi}^{l}\bar{\psi}^{(l)}(1_{\tilde{A}})^{\frac{1}{2}})(1_4\otimes \bar{\iota}\iota(a)))\\
& = & \sum_{l=0}^{m+1} e_{1, l} \otimes (e_{1, 1} \otimes \bar{\gamma}(\iota(a)))(e_{1, 1} \otimes \bar{\gamma}(\bar{\phi}^{(l)}\bar{\psi}^{(l)}(1_{\tilde{A}})^{\frac{1}{2}}\bar{s}^{(l)}) \quad\quad\textrm{(by \eqref{twrist})}\\
& = & (1_{m+2}\otimes e_{1, 1} \otimes\bar{\gamma}(\iota(a))) \sum_{l=0}^{m+1} e_{1, l} \otimes e_{1, 1} \otimes \bar{\gamma}(\bar{\phi}^{(l)}\bar{\psi}^{(l)}(1_{\tilde{A}})^{\frac{1}{2}}\bar{s}^{(l)}) \\
& = & (1_{m+2}\otimes e_{1, 1} \otimes \bar{\gamma}(\iota(a)))\bar{v},
\end{eqnarray*}
and hence
\begin{equation*}
\bar{v}^*\bar{v} (1_{m+2}\otimes 1_4 \otimes \bar\iota\iota(a))  =  \bar{v}^*(1_{m+2}\otimes e_{1, 1} \otimes \bar\gamma\iota(a)) \bar{v} =  (1_{m+2}\otimes 1_4\otimes\bar\iota\iota(a)) \bar{v}^*\bar{v},\quad a\in \tilde{A}.
\end{equation*}

Then, for any finite set $\mathcal F\subseteq \tilde{A}$ and any $\eps>0$, there are $i, j\in\mathbb N$ and $v\in \mathrm{M}_{m+2}(\Comp)\otimes \mathrm{M}_{4}(\Comp) \otimes \tilde{A}$ such that
\begin{enumerate}
\item $vv^*= e_{1, 1}\otimes e_{1, 1}\otimes \tilde{\rho}_i(1_{\tilde{S_i}})= e_{1, 1}\otimes e_{1, 1}\otimes 1_{\tilde{A}}$,
\item $\|[v^*v, 1_{m+2} \otimes a]\| < \eps$, $a\in \mathcal F$, 
\item $\|v^*v(1_{m+2}\otimes a) - (e_{1, 1}\otimes\tilde{\rho}_i\tilde{\sigma}_i(a))v\|<\eps$, $a\in\mathcal F$.
\end{enumerate}

Define
$$\kappa: \tilde{S}_i \to \mathrm{M}_{m+2}(\Comp)\otimes \mathrm{M}_{4}(\Comp) \otimes \tilde{A}$$
by
$$\kappa(s)=v^*(e_{1, 1}\otimes e_{1, 1} \otimes\tilde{\rho}_i(s))v.$$
Note that
$$\kappa(S_i) \subset \mathrm{M}_{m+2}(\Comp)\otimes \mathrm{M}_{4}(\Comp) \otimes A.$$

Then $\kappa$ is an embedding; and on setting $p=1_{\kappa(\tilde{S}_i)}=v^*v$, one has
\begin{enumerate}
\item $p \sim e_{1, 1}\otimes e_{1, 1}\otimes 1_{\tilde A}$,
\item $\|[p, 1_{m+2}\otimes 1_4 \otimes a]\| <\eps$, $a\in\mathcal F$,
\item $p(1_{m+2}\otimes 1_4 \otimes a)p \in_\eps \kappa(\tilde{S}_i)$, $a\in\mathcal F$.
\end{enumerate}

Since $p\sim e_{1, 1}\otimes e_{1, 1}\otimes 1_{\tilde A}$, one has that $ (1-p)\mathrm{M}_{4(m+2)}(A)(1-p) \cong \mathrm{M}_{4(m+2)-1}(A)$, and hence is still in the reduction class $\mathcal R$, as desired.
%
%
\end{proof}

\begin{lem}\label{pre-decp}
Let $A$ be a simple C*-algebra satisfying the UCT. Assume $\mathrm{dim}_{\mathrm{nuc}}(A)=m<\infty$, $\Kzero(A)=\{0\}$, and $A\in\mathcal R$. Then, for any finite set $\mathcal F \subset A$, and $\eps>0$, there are $d\in\mathbb N$, a projection $p$ in $\mathrm{M}_d(\tilde{A})$, and a sub-C*-algebra $S\subset p\mathrm{M}_d(A)p$ with $S\in\mathcal R_z$, such that
\begin{enumerate}
\item $\|[p, 1_d\otimes a ]\| < \eps$, $a\in\mathcal F$,
\item $p(1_d\otimes a)p\in_\eps S$, $a\in\mathcal F$, and
\item $1-p$ is Murray von Neumann equivalent to a projection of $\mathrm{M}_d(\Comp1_{\tilde{A}})$ with (normalized) trace at most $\eps$.
\end{enumerate}
\end{lem}

\begin{proof}
Without loss of generality, let us assume that $\|a\|\leq 1$, $a\in\mathcal F$.

Pick $N\in\mathbb{N}$ such that 
$$(\frac{4(m+2)-1}{4(m+2)})^N<\eps.$$ Set $\eps'=\frac{\eps}{8N}$.

By Theorem \ref{TFA}, the C*-algebra $A$ satisfies the conditions of Theorem \ref{fdim}, and hence there is a projection $p_1\in \mathrm{M}_{4(m+2)}(\tilde A)$ and a sub-C*-algebra $S_1\in p\mathrm{M}_{4(m+2)}(\tilde A)p$ with $S_1\in\mathcal R_0$ such that 
\begin{enumerate}
\item $\|[p_1, 1_{4(m+2)} \otimes a]\| < \eps'$, $a\in\mathcal F$,
\item $p_1(1_{4(m+2)} \otimes a) p_1\in_{\eps'} S_1$, $a\in\mathcal F$,
\item $p_1\sim e_{11}$ in $\mathrm{M}_{4(m+2)}(\tilde{A})$, and
\item $(1-p_1)\mathrm{M}_{4(m+2)}(A)(1-p_1) \in \mathcal R$.
\end{enumerate}

Consider the sub-C*-algebra $A_1:=(1-p_1)\mathrm{M}_{4(m+2)}(A)(1-p_1)$. Its nuclear dimension is bounded by $m$ and it also satisfies the UCT. Thus, by Theorem \ref{TFA} and Theorem \ref{fdim}, there is a projection $p_2\in \mathrm{M}_{4(m+2)}(\tilde A_1)\subset \mathrm{M}_{4(m+2)}((1-p_1)\mathrm{M}_{4(m+2)}(\tilde{A})(1-p_1))$ and a sub-C*-algebra $S_2\in p_2\mathrm{M}_{4(m+2)}(\tilde A)p_2$ with $S_2\in\mathcal R_0$ such that 
\begin{enumerate}
\item $\|[p_2, 1_{4(m+2)} \otimes a]\| < \eps'$, $a\in(1-p_1)(\mathrm M_{4(m+1)}\otimes\mathcal F)(1-p_1)$,
\item $p_2(1_{4(m+2)} \otimes a) p_2\in_{\eps'} S_2$, $a\in(1-p_1)(\mathrm M_{4(m+1)}\otimes\mathcal F)(1-p_1)$,
\item $p_2\sim e_{11}$ in $\mathrm{M}_{4(m+2)}(\tilde{A_1})\subset \mathrm{M}_{4(m+2)}((1-p_1)\mathrm{M}_{4(m+2)}(\tilde{A})(1-p_1))$, and
\item $(1-p_2)\mathrm{M}_{4(m+2)}(A_1)(1-p_2) \in \mathcal R$.
\end{enumerate}

Note that $p_1 \perp p_2$ as projections of $\mathrm M_{4(n+2)}\otimes \mathrm M_{4(n+2)}\otimes \tilde{A}$ (where $p_1$ is regarded as $1_{4(m+2)}\otimes p_1$), and note that
\begin{eqnarray*}
& &1_{4(m+2)}\otimes 1_{4(m+2)}\otimes a \\
 &\approx_{2\eps'}& p_1(1_{4(m+2)}\otimes 1_{4(m+2)}\otimes a)p_1 + (1-p_1)(1_{4(m+2)}\otimes 1_{4(m+2)}\otimes a)(1-p_1) \\
&\approx_{2\eps'}& p_1(1_{4(m+2)}\otimes 1_{4(m+2)}\otimes a)p_1 + p_2((1-p_1)(1_{4(m+2)}\otimes 1_{4(m+2)}\otimes a)(1-p_1))p_2 \\
&&+(1-p_2)((1-p_1)(1_{4(m+2)}\otimes 1_{4(m+2)}\otimes a)(1-p_1))(1-p_2),
\end{eqnarray*}
and for any $a\in\mathcal F$,
$$p_1(1_{4(m+2)}\otimes 1_{4(m+2)}\otimes a)p_1 \in_{\eps'} S_1\quad\mathrm{and}\quad p_2((1-p_1)(1_{4(m+2)}\otimes 1_{4(m+2)}\otimes a)(1-p_1))p_2 \in_{\eps'} S_2.$$
Therefore, 
\begin{enumerate}
\item $\|[p_1+p_2, 1_{4(m+2)}\otimes 1_{4(m+2)} \otimes a]\| < 8\eps'$, $a\in \mathcal F$,
\item $(p_1+p_2)(1_{4(m+2)}\otimes 1_{4(m+2)} \otimes a) (p_1+p_2)\in_{4\eps'} S_1\oplus S_2$, $a\in \mathcal F$.
\end{enumerate}

Also note that $p_1+p_2 $ is Murray-von Neumann equivalent to a projection of $\mathrm M_{4(m+2)}(\mathrm M_{4(m+2)}(\Comp 1_{\tilde{A}}))$ with rank $4(m+2) + 4(m+2)-1$, and 
$$\frac{\mathrm{rank}(1-(p_1+p_2))}{(4(m+2))^2} = (\frac{4(m+2) - 1}{4(m+2)})^2.$$

%

Then, repeating this procedure $N$ times, one obtains pairwise orthogonal projections $$p_1, p_2, ..., p_N \in \underbrace{\mathrm{M}_{4(m+2)}(\Comp)\otimes\cdots\otimes \mathrm{M}_{4(m+2)}}_N(\Comp)\otimes {\tilde{A}},$$ and sub-C*-algebras $S_i\in p_i ((\bigotimes_N \mathrm M_{4(m+2)}(\Comp))\otimes{\tilde{A}})p_i$
such that each $p_i$ is Murray-von Neumann equivalent to a projection of $\bigotimes_N \mathrm M_{4(m+2)}(\Comp 1_{\tilde{A}})$ and $$\frac{\mathrm{rank}(1-(p_1+\cdots+p_N))}{(4(m+2))^N} = (\frac{4(m+2)-1}{4(m+2)})^N <\eps,$$ and
\begin{enumerate}
\item $\|[p_1+\cdots+p_N, 1_{(4(m+2))^N} \otimes a]\| < 8(N-1)\eps'<\eps$, $a\in \mathcal F$,
\item $(p_1+\cdots + p_N)(1_{(4(m+2))^N}\otimes a) (p_1+\cdots + p_N)\in_{4(N-1)\eps'} S_1\oplus\cdots\oplus S_N$, $a\in \mathcal F$.
\end{enumerate}
Thus, one obtains the desired decomposition with $d=(4(m+2))^N$, $p=p_1+\cdots+p_N$, and $S=S_1\oplus\cdots\oplus S_N$.
\end{proof}

The following statement is a straightforward corollary of Lemma \ref{pre-decp}.
\begin{cor}\label{pre-div-decp}
Let $A$ be a simple C*-algebra satisfying the UCT. Assume that $\mathrm{dim}_{\mathrm{nuc}}(A)=m<\infty$, $\Kzero(A)=\{0\}$, and $A\in\mathcal R$. Then, for any finite set $\mathcal F\subset A$, any $\eps>0$ and any $n\in \mathbb N$, there is a projection $p \in \tilde{A}\otimes Q$ and a sub-C*-algebra $S\subset p(A\otimes Q)p$ with $S\in\mathcal R_z$ such that
\begin{enumerate}
\item $\|[p, a\otimes 1_Q]\| < \eps$, $a\in\mathcal F$,
\item $p(a\otimes 1_Q)p \in_{\eps} S$, $a\in\mathcal F$,
\item $1-p$ is Murray von Neumann equivalent to a projection of $1_{\tilde{A}} \otimes Q$ with trace at most $\eps$.
\end{enumerate}
\end{cor} 

One then has the following (tracial) divisibility property.
\begin{thm}\label{div-decp}
Let $A$ be a simple C*-algebra satisfying the UCT. Assume that $\mathrm{dim}_{\mathrm{nuc}}(A)=m<\infty$, $\Kzero(A)=\{0\}$, and $A\in\mathcal R$. Then, for any finite set $\mathcal F\subset A$, any $\eps>0$ and any $n\in \mathbb N$, there are projections $q, p_1, p_2, ..., p_n \in \tilde{A}\otimes Q$, $S\subset p_1(A\otimes Q)p_1$, $S\in\mathcal R_z$ such that
\begin{enumerate}
\item $q+p_1+p_2+\cdots+p_n =1_{\tilde{A}\otimes Q}$, $q$ is Murray-von Neumann equivalent to $p_i$, $i=1, 2, ..., n$,
\item $\|[p_i, a\otimes 1_Q]\| < \eps$, $a\in\mathcal F$,
\item $p_iap_i\in_{\eps}v_iSv^*_i$, $a\in\mathcal F$, where $v_i\in \tilde{A}\otimes Q$, $i=1, 2, ..., n$, are the partial isometries  satisfying $v^*_iv_i=p_1$ and $v_iv_i^*=p_i$.
\end{enumerate}
\end{thm}

\begin{proof}
Let $(\mathcal F, \eps)$, and $n$ be given. It follows from Corollary \ref{pre-div-decp} that there exist a projection $p \in \tilde{A}\otimes Q$ and a sub-C*-algebra $S'\subset p(A\otimes Q)p$ with $S\in\mathcal R_0$ such that
\begin{enumerate}
\item $\|[p, a\otimes 1_Q]\| < \eps$, $a\in\mathcal F$,
\item $p(a\otimes 1_Q)p \in_{\eps} S'$, $a\in\mathcal F$,
\item $1-p$ is Murray von Neumann equivalent to a projection of $1_{\tilde{A}} \otimes Q$ with trace $r$ and $r < \frac{1}{n}$.
\end{enumerate}

Choose natural numbers $l$ and $m$ such that $m$ and $l$ are divided by $n$ and
\begin{equation}\label{size-mul}
\frac{l}{m} = (\frac{1}{n+1} - r)\cdot \frac{1}{1-r}.
\end{equation}
Then, consider $\tilde{A}\otimes Q \otimes \mathrm{M}_{m}$, and consider 
$$q:=(1-p)\otimes 1_m + p\otimes e_{1, 1} + p\otimes e_{2, 2} +\cdots + p\otimes e_{l, l}.$$ Then $q$ is Murray-von Neumann equivalent to a projection of $1_{\tilde{A}}\otimes Q \otimes \mathrm{M}_{m}$ with trace
$$r+(1-r)\frac{l}{m}=\frac{1}{n+1},\quad \textrm{by \eqref{size-mul}}.$$

Since $m-l$ is divided by $n$, there are pairwise orthogonal projections $e_1, e_2, ..., e_n\in \mathrm{M}_m(\Comp)$  such that $$e_i\perp (e_{1, 1}+\cdots+e_{ll}),\quad\mathrm{and}\quad e_i\sim e_j,\quad i, j=1, 2, ..., n,$$
Then put
$$p_i=p\otimes e_i\in \tilde{A}\otimes Q \otimes \mathrm{M}_{m},\quad i=1, 2, ..., n,$$
and 
$$S=S'\otimes e_i\subset p_1(A\otimes Q\otimes\mathrm{M}_m)p_1.$$ Note that there are partial isometries $v_i\in 1_{\tilde{A}\otimes Q}\otimes \mathrm{M}_m$, $i=1, 2, ..., n$,  satisfying $v^*_iv_i=p_1$ and $v_iv_i^*=p_i$. Also note that each $p_i$, $i=1, 2, ..., n$, is Murray-von Neumann equivalent to a projection of $1_{\tilde{A}}\otimes Q\otimes \mathrm M_m$ with trace $\frac{1}{n+1}$, and, therefore, $p_i$ is murray-von Neumann equivalent to $q$.

 Regard $\tilde{A}\otimes Q \otimes \mathrm{M}_{m}$ as a unital sub-C*-algebra of $\tilde{A}\otimes Q \otimes Q$, and identify $Q\otimes Q$ with $Q$.
Then $q, p_1, p_2, ..., p_n$, $v_1, v_2, ..., v_n$, and $S$ satisfy the conclusion of the theorem.
\end{proof}

\begin{rem}
It follows from Theorem \ref{div-decp} that any element $a\in A$ can be approximated by
$$qaq + p_1ap_1 +p_2ap_2 + \cdots + p_nap_n$$
as an diagonal element of $\mathrm{M}_{m+1}(qAq)$. This structure theorem will lead to a uniqueness theorem (Theorem \ref{uniq-thm}), and then a classification of C*-algebras with trivial K-theory (Theorem \ref{clas-thm}).
\end{rem}

\begin{cor}\label{decp}
Let $A$ be a C*-algebra of Theorem \ref{div-decp}. Assume that $A\cong A\otimes Q$. Then, for any finite set $\mathcal F\subset A$, any $\eps>0$, and any $n\in\mathbb{N}$,  there are sub-C*-algebras $A_0, A_1\subset A$, $S\subset A_1$, with $A_0\perp A_1$ and $S\in\mathcal R_z$, and c.p.~maps
$$
\xymatrix{
A \ar[r]^-{\theta=\theta_0\oplus\theta_1} & A_0 \oplus A_1 \ar@{^(->}[r]^-\iota & A
}
$$
such that $\theta$ is $\mathcal F$-$\eps$-multiplicative, $\iota$ is the embedding, and
\begin{enumerate}
\item $\|\iota(\theta(a)) - a\| < \eps$, $a\in\mathcal F$,
\item $\theta_1(a)\in_\eps S$, $a\in\mathcal F$,
\item there are $v_1, v_2, ..., v_n\in A$ such that
         \begin{enumerate}
         \item $\|(v_i^*v_i)\theta_0(a) -\theta_0(a)\| < \eps$, $a\in\mathcal F$, 
         \item $v_iv_i^* \in A_1$ and  $v_iv_i^* \perp v_jv_j^*$, $i\neq j$.
         \end{enumerate}
\end{enumerate}
\end{cor}

Motivated by this, one has the following notion of tracial approximation for a C*-algebra with or without unit.
\begin{defn}
Let $\mathcal S$ be a class of C*-algebras. A C*-algebra $A$ is said to be tracially approximated by C*-algebras in $\mathcal S$, denoted by $A\in\mathrm{TA}\mathcal S$, if 
any finite set $\mathcal F\subset A$, any $\eps>0$, and any $n\in\mathbb N$, there are sub-C*-algebras $A_0, A_1 \subset A$, $S\subset A_1$, with $A_0 \perp A_1$ and $S\in\mathcal S$, and c.p.~maps
$$
\xymatrix{
A \ar[r]^-{\theta=\theta_0\oplus\theta_1} & A_0 \oplus A_1 \ar@{^(->}[r]^-\iota & A
}
$$
such that $\theta$ is $\mathcal F$-$\eps$-multiplicative, $\iota$ is the natural embedding, and
\begin{enumerate}
\item $\|\iota(\theta(a)) - a\| < \eps$, $a\in\mathcal F$,
\item $\theta_1(a)\in_\eps S$, $a\in\mathcal F$,
\item there are $v_1, v_2, ..., v_n\in A$ such that
         \begin{enumerate}
         \item $\|(v_i^*v_i)\theta_0(a) -\theta_0(a)\| < \eps$, $a\in\mathcal F$, 
         \item $v_iv_i^* \in A_1$ and  $v_iv_i^* \perp v_jv_j^*$, $i\neq j$.
         \end{enumerate}
\end{enumerate}
\end{defn}

\begin{lem}
Let $A$ be a simple C*-algebra with strict comparison of positive elements, and assume that $A$ has stable rank one. Then $A\in\mathrm{TA}\mathcal S$ if, and only if, 
for any finite set $\mathcal F\subset A$, and any $\eps>0$, there are sub-C*-algebras $A_0, A_1 \subset A$, $S\subset A_1$, with $A_0 \perp A_1$ and $S\in\mathcal S$, and c.p.~maps
$$
\xymatrix{
A \ar[r]^-{\theta=\theta_0\oplus\theta_1} & A_0 \oplus A_1 \ar@{^(->}[r]^-\iota & A
}
$$
such that $\theta$ is $\mathcal F$-$\eps$-multiplicative, $\iota$ is the embedding, and
\begin{enumerate}
\item $\|\iota(\theta(a)) - a\| < \eps$, $a\in\mathcal F$,
\item $\theta_1(a)\in_\eps S$, $a\in\mathcal F$,
\item $\tau(b) <\eps $, $b\in A_0$ with $\|b\|\leq 1$, $\tau\in\mathrm{T}(A)$, $\|\tau\| =1$.
\end{enumerate}
\end{lem}

\section{Classification of KK-contractible C*-algebras}
The structure (Theorem \ref{div-decp}) obtained in the previous section actually leads us to a classification of KK-contractible C*-algebras. One first need a stable uniqueness theorem with codomain algebra not necessary unital.

\subsection{A stable uniqueness theorem}\label{section-stable-uniq}

Let $A, B$ be C*-algebras, and let $\gamma: A \to B$ be a homomorphism, and consider the homomorphism $$\gamma_\infty:=\mathrm{diag}\{\gamma, \gamma, ...\}: A \to\mathcal M(\mathcal K\otimes B),$$
where $\mathcal K$ is the algebra of compact operates on a separable Hilbert space, and $\mathcal M(\mathcal K\otimes B)$ is the multiplier algebra.
\begin{lem}[\cite{EK-BDF}]\label{absorbing}
Assume that $A$ is not unital, and assume that $B$ is separable. The extension $\gamma_\infty$ is absorbing, in the nuclear sense, if it is full, i.e., $\overline{B\gamma(a)B}=B$, $a\in A\setminus\{0\}$.
\end{lem}
\begin{proof}
By Theorem 17 (iii) of \cite{EK-BDF}, the extension $\gamma_\infty$ is purely large. Note that $A$ does not have a unit. Then it follows from Corollary 16 of \cite{EK-BDF} (see Theorem 2.1 \cite{Gabe-BDF} for a correction) that $\gamma_\infty$ is nuclear absorbing.
\end{proof}

\begin{prop}\label{no-1-uniq-hom-0}
Let $A$ be a separable C*-algebra without a unit, and let $B$ be a separable C*-algebra. Let $\gamma: A\to B$ be a full homomorphism, i.e., $\overline{B\gamma(a)B}=B$, $a\in A\setminus\{0\}$.

Suppose that $\phi, \psi: A\to B$ are two nuclear homomorphisms with $[\phi]=[\psi]$ in $\mathrm{KK}_\mathrm{nuc}(A, B)$. Then for any finite set $\mathcal F\subset A$ and $\eps>0$, there exist an integer $n$ and a unitary $u\in {\mathrm{M}_{n+1}(B)}+\Comp 1_{n+1}$ satisfying 
$$\|u^*(\phi(a)\oplus(\underbrace{\gamma(a)\oplus\cdots\oplus\gamma(a)}_n)u - \psi(a)\oplus(\underbrace{\gamma(a)\oplus\cdots\oplus\gamma(a)}_n)\| < \eps,\quad a\in\mathcal F.$$
\end{prop}

\begin{proof}
Since $[\phi] = [\psi]$ in $\mathrm{KK}_\mathrm{nuc}(A, B)$, there is a strict nuclear representation $\sigma: A \to \mathcal M(\mathcal K\otimes B)$ such that $\phi\oplus\sigma \cong \psi\oplus\sigma$, where $``\cong"$ is the properly asymptotically unitary equivalence define in \cite{DE-KK-Asy}. Hence $\phi\oplus\sigma\oplus\gamma_\infty \cong \psi\oplus\sigma\oplus\gamma_\infty$. Since $\gamma_\infty$ is (non-unital) absorbing, one has $\sigma\oplus\gamma_\infty\sim\gamma_\infty$. It then follows from Lemma 4.3 of \cite{DL-classification} that there is a sequence of unitaries $u_k \in\mathcal K\otimes B + \Comp 1$ that 
$$\|u_k^*(\phi(a)\oplus \gamma_\infty(a))u_k - \psi(a)\oplus \gamma_\infty(a)\|\to 0,\quad a\in A,\ k\to\infty.$$

Let $e_n=\mathrm{diag}\{\underbrace{1_{\tilde{B}}, 1_{\tilde{B}}, ..., 1_{\tilde{B}}}_n, 0, ...\}$. Since $u_k \in \mathcal K\otimes B + \Comp 1$, one has that $[u_k, e_n]\to 0$, as $n\to\infty$. Then, for sufficiently large $k$, and then a sufficiently large $n$, the element $e_nu_ke_n \in{\mathrm{M}_{n}(B)}+\Comp 1_{n}$ can be perturbed to a unitary $u$ satisfying the conclusion of the proposition.
\end{proof}

In order to get a stable uniqueness theorem for approximate homomorphisms and with a control on the multiplicity $n$, we shall apply Proposition \ref{no-1-uniq-hom-0} to the case that $B$ is an asymptotical sequence algebra, i.e., $B=\prod B_n/\bigoplus B_n$ for some C*-algebras $B_n$. Such an asymptotical sequence algebra is not $\sigma$-unital if $B_n$ are not unital. This brings a trouble for the picture of $\mathrm{KK}(A, B)$ in general. In the following, we shall show that in this situation ($B$ is not $\sigma$-unital), there is a separable sub-C*-algebra $B'\subset B$ containing the images of $A$ under $\phi, \psi,$ and $\gamma$, and $B'$ is also large enough so that the map $\gamma: A \to B'$ is still full. Then one just works on the maps from $A$ to $B'$ instead.

\begin{prop}[Proposition 2.1 of \cite{Bla-WEP}]\label{lem-of-bla}
Let $A$ be a separable C*-algebra (with or without unit). Then there is a countable subset $S$ of $A$ such that if $J$ is any ideal of $A$, then $S\cap J$ is dense in $J$.
\end{prop}

\begin{lem}\label{sep-cond}
Let $D$ be a C*-algebra. Let $A\subseteq D$ be separable sub-C*-algebra such that 
$$\overline{DaD}=D, \quad a\in A\setminus\{0\},$$ and let $B\subseteq D$ be another separable sub-C*-algebra. Then, there is a separable sub-C*-algebra $C$ of $D$ such that
$$A, B \subseteq C$$
and
$$\overline{CaC}=C, \quad a\in A\setminus\{0\}.$$
\end{lem}
\begin{proof}
Apply Proposition \ref{lem-of-bla}, one obtains a  countable set $$\{a_0, a_1, a_2, ...\}\subset A$$ such that
$\{a_0, a_1, a_2, ...\}\cap J$ is dense in $J$ for any ideal $J$ of $A$. We may assume that $$a_0=0\quad \textrm{and}\quad a_j\neq 0,\quad j=1, 2, ...\ .$$

Set $$C_1=\textrm{C*}\{A, B\}\subseteq D.$$ It is clear that $C_1$ is separable. Pick a dense set $\{c_1, c_2, ...\}$ in $C_1$. Since $\overline{Da_jD}=D$, $j=1, 2, ...$, for any $\eps>0$ and any $c_i$, there are eventually zero sequences $x_{c_i, a_j, \eps, 1}, x_{c_i, a_j, \eps, 2}, ...,$ and $y_{c_i, a_j, \eps, 1}, y_{c_i, a_j, \eps, 2}, ...,$ in $D$ such that 
and
$$ \| c_i - (x_{c_i, a_j, \eps, 1} a_j y_{c_i, a_j, \eps, 1}+  x_{c_i, a_j, \eps, 2}a_jy_{c_i, a_j, \eps, 2}+ \cdots )\|<\eps.$$

Set $$C_2=\textrm{C*}\{C_1, x_{c_i, a_j, \frac{1}{n}, k}, y_{c_i, a_j, \frac{1}{n}, k} : i, j, n, k=1, 2, ... \},$$ and, then one has that 
$$\overline{C_2a_jC_2} \subseteq C_1,\quad j=1, 2, ...\ .$$

Repeat the construction above, one obtains a sequence of separable C*-algebras $$C_1\subseteq C_2\subseteq \cdots \subseteq C_n \subseteq \cdots\subseteq D$$
such that 
$$\overline{C_{n+1}a_jC_{n+1}} \subseteq C_n,\quad j=1, 2, ...,\ n=1, 2, ...\ .$$
Set $C=\overline{\bigcup_{n=1}^\infty C_n}$, and then one has 
$$\overline{Ca_jC} = C,\quad j=1, 2, ...,\ .$$

Then the separable sub-C*-algebra $C$ satisfies the lemma. Indeed, let $a \in A\setminus\{0\}$. Consider the ideal $J:=\overline{CaC} \cap A$. Since $a\in J$, one has that $J\neq\{0\}$. By Proposition \ref{lem-of-bla}, one has that $\{a_0, a_1, a_2, ...\}\cap J$ is dense in $J$, and in particular, the ideal $J$ contains some $a_j\neq 0$. Since 
$C=\overline{Ca_jC} \subseteq \overline{CJC}=\overline{CaC}$, one has that $\overline{CaC}=C$, as desired.
\end{proof}

\begin{rem}
If $A$ is simple, then, in the proof above, one only needs to pick one non-zero element of $A$ and does not need Proposition \ref{lem-of-bla}. 
\end{rem}

Let us point out the following corollary, although it is not used in this paper.
\begin{cor}
Let $A, D$ be C*-algebra, and assume that $A$ is separable. Let $\phi, \psi, \sigma: A \to D$ be homomorphisms such that 
\begin{enumerate}
\item $[\phi] = [\psi]$ in $\mathrm{Hom}(\underline{\mathrm{K}}_*(A), \underline{\mathrm{K}}_*(D))$, and
\item $\overline{D\sigma(a)D} = D$, $a\in A\setminus\{0\}$. 
\end{enumerate}
Then there is a separable sub-C*-algebra $C\subset D$ such that
\begin{enumerate}
\item $\phi(A), \psi(A), \sigma(A)\subset C$,
\item $[\phi] = [\psi]$ in $\mathrm{Hom}(\underline{\mathrm{K}}_*(A), \underline{\mathrm{K}}_*(C))$,
\item $\overline{C\sigma(a)C} = C$.
\end{enumerate}
\end{cor}
\begin{proof}
Since $A$ is separable (and hence all $\mathrm{K}_*(A, \Int/n\Int)$, $*=0, 1$, $n\in\mathbb N$, are countable), there is a separable sub-C*-algebra $B\subset D$ such that $\phi(A), \psi(A)\subset B$ and
$[\phi] = [\psi]$ in $\mathrm{Hom}(\underline{\mathrm{K}}_*(A), \underline{\mathrm{K}}_*(B))$. The the existence of $C$ follows directly from 
Lemma \ref{sep-cond}.
\end{proof}

\begin{defn}
Let $L: A^+ \times (0, 1) \to \mathbb N$ and $N: A^+ \times (0, 1) \to (0, +\infty)$ be maps.

A positive map $\phi: A\to B$ is said to be $(L, N)$-full if for any $\eps>0$, any $a\in A^+\setminus\{0\}$, and any $b\in B^+$ with $\|b\| \leq 1$, there are $b_1, b_2, ..., b_{L(a, \eps)}$ with $\|b_i\| \leq N(a, \eps)$, $i=1, 2, ..., L(a, \eps)$, such that
$$\|b- (b_1\phi(a)b^*_1 + b_2\phi(a)b^*_2 + \cdots + b_{L(a, \eps)}\phi(a)b^*_{L(a, \eps)})\|<\eps.$$
\end{defn}

\begin{defn}
A C*-algebra $A$ is said to be KK-contractible if $A$ is KK-equivalent to $\{0\}$, or equivalently, if $\mathrm{KK}(A, A)=\{0\}$.
\end{defn}

\begin{lem}\label{KK-Kun}
If $B$ is KK-contractible, then $A\otimes B$ is KK-contractible for any nuclear C*-algebra $A$.
\end{lem}
\begin{proof}
Since $B$ is KK-contractible, one has $\mathrm{id}_{B} \sim_{\mathrm{KK}} 0_{B}$, and hence there is a continuous path (in the strict topology) of pairs $(\phi^+_t, \phi^-_t)$, $t\in[0, 1]$, where $$\phi^{\pm}_t: B \to \mathcal M(B\otimes\mathcal K),\quad t\in[0, 1],$$
are homomorphisms such that 
$$\phi_t^+(a)-\phi_t^-(a) \in B\otimes\mathcal K,\quad t\in[0, 1],\ a\in B,$$
$$(\phi_0^+, \phi_0^-) = (\mathrm{id}_{B}, 0)\quad\mathrm{and}\quad (\phi_1^+, \phi_1^-) = (0, 0).$$

Fix a nuclear C*-algebra $A$. Consider 
$$\Phi_t^{\pm}(a\otimes b)=a\otimes\phi_t^{\pm}(b)\in A\otimes \mathcal M( B\otimes\mathcal K)\subset \mathcal M(A\otimes B\otimes\mathcal K),\quad a\in A,\ b\in B,\ t\in[0, 1].$$
Then $\Phi_t^{\pm}(a\otimes b)$, $t\in[0, 1]$, are continuous paths (with the strict topology) in $\mathcal M(A\otimes B \otimes\mathcal K)$, and
$$\Phi_t^+(a\otimes b) - \Phi_t^-(a\otimes b) = a\otimes (\phi_t^+(b) - \phi_t^-(b)) \in A\otimes B \otimes\mathcal K.$$
Moreover, 
$(\Phi^+_0, \Phi^-_0) = (\mathrm{id}_{A \otimes \mathcal W}, 0)$ and $(\Phi^+_1, \Phi^-_1) = (0, 0)$. Therefore, $\mathrm{id}_{A\otimes B} \sim_{\mathrm{KK}} 0$, and $A\otimes B$ is KK-contractible, as desired.
\end{proof}

\begin{prop}\label{no-1-uniq-hom-1}
Let $A$ be a separable nuclear C*-algebra without unit which is KK-contracible. Let $L: A^+ \times (0, 1) \to \mathbb N$ and $N: A^+ \times (0, 1) \to (0, +\infty)$ be maps.

Let $\mathcal F\subset A$ be a finite set, and let $\eps>0$. There exists $n\in\mathbb N$ such that for any homomorphisms $\phi, \psi, \sigma: A \to B$, where $B$ is a C*-algebra, and $\sigma$ is $L$-$N$ full, there is a unitary $u\in{\mathrm{M}_{n+1}(B)}+\Comp 1_{n+1}$ such that
$$\| u(\phi(a)\oplus\underbrace{\sigma(a)\oplus\cdots\oplus\sigma(a)}_n)u^* - \psi(a)\oplus\underbrace{\sigma(a)\oplus\cdots\oplus\sigma(a)}_n\| <\eps,\quad a\in\mathcal F.$$
\end{prop}
\begin{proof}
Assume the statement were not true. There are $(\mathcal F, \eps)$ such that for each $n\in\mathrm{N}$, there is a C*-algebra $B_n$ and homomorphisms $\phi_n, \psi_n, \sigma_n: A\to B_n$ such that
$$\max_{a\in\mathcal F}\inf_{u\in\mathrm U({\mathrm{M}_{n+1}(B_n)}+\Comp 1_{n+1})}\| u(\phi(a)\oplus\underbrace{\sigma(a)\oplus\cdots\oplus\sigma(a)}_n)u^* - \psi(a)\oplus\underbrace{\sigma(a)\oplus\cdots\oplus\sigma(a)}_n\| \geq \eps.$$

Consider the maps $\Phi$, $\Psi$, and $\Sigma: A \to \prod B_n/\bigoplus B_n$ induced by $(\phi_n)$, $(\psi_n)$, and $(\sigma_n)$ respectively. Since $\sigma_n$, $n=1, 2, ...$, are $L$-$N$ full, the map $\Sigma$ is full. Since $A$ is separable, it follows from Lemma \ref{sep-cond} that there is a separable C*-algebra $B\subset \prod B_n/\bigoplus B_n$ such that 
$$\Phi(A), \Psi(A), \Sigma(A)\subset B$$
and the map $\Sigma: A \to B\subset \prod B_n/\bigoplus B_n$ is full.

Then it follows from Proposition \ref{no-1-uniq-hom-0} that there is $k\in\mathbb{N}$ and a unitary $$U\in {\mathrm{M}_{k+1}(B)}+\Comp 1_{k+1}\subset \mathrm{M}_{k+1}(\prod B_n/\bigoplus B_n) +\Comp 1_{k+1} $$ such that
$$\| U(\Phi(a)\oplus\underbrace{\Sigma(a)\oplus\cdots\oplus\Sigma(a)}_k)U^* - \Psi(a)\oplus\underbrace{\Sigma(a)\oplus\cdots\oplus\Sigma(a)}_k\| <\eps,\quad a\in\mathcal F.$$

By lifting $U$ to a unitary $(u_n) \in\mathrm{M}_{k+1}(\prod{B_n})+\Comp1_{k+1}$, one obtains a contradiction.
\end{proof}

\begin{rem}
If $A$ is simple, the condition that $\sigma$ is $L$-$N$-full can be replaced by $\sigma(a)$ is $L(a, \eps)$-$N(a, \eps)$-full for a previously fixed $a$.
\end{rem}

\begin{prop}\label{stable-uniq-full}
Let $A$ be a separable nuclear C*-algebra without unit which is KK-contractible. Let $L: A^+ \times (0, 1) \to \mathbb N$ and $N: A^+ \times (0, 1) \to (0, +\infty)$ be maps.

Let $\mathcal F\subset A$ be a finite set, and let $\eps>0$. There exist $\mathcal G\subset A$, $\mathcal H\subset A^+$, $\delta>0$, $n\in\mathbb N$ such that for any $\mathcal G$-$\delta$-multiplicative maps $\phi, \psi, \sigma: A \to B$ (where $B$ is a C*-algebra) satisfying
$$\textrm{$\sigma(a)$ is $L(a, \eps)$-$N(a, \eps)$ full,\quad $\eps>0$,\ $a\in\mathcal H$},$$
there is a unitary $u\in{\mathrm{M}_{n+1}(B)}+\Comp 1_{n+1}$ such that
$$\| u(\phi(a)\oplus\underbrace{\sigma(a)\oplus\cdots\oplus\sigma(a)}_n)u^* - \psi(a)\oplus\underbrace{\sigma(a)\oplus\cdots\oplus\sigma(a)}_n\| <\eps,\quad a\in\mathcal F.$$
\end{prop}
\begin{proof}
Assume the statement were not true. There is a pair $(\mathcal F, \eps)$ such that for any $n$, any finite sets $\mathcal G, \mathcal H\subset A$, any constant $\delta$, there is a C*-algebra $D$ and $\mathcal G$-$\delta$-multiplicative maps $\phi, \psi, \sigma$ satisfying 
\begin{enumerate}
\item $[\phi(p)]=[\psi(p)]$ for all $p\in\mathcal P$,
\item $\sigma(a)$ is $L(a, \eps)$-$N(a, \eps)$-full for any $a \in \mathcal H$ and $\eps>0$,
\end{enumerate}
but
$$\max_{c\in\mathcal F}\|u(\phi(c)\oplus(\bigoplus^n \sigma(c)))u^* - \psi(c)\oplus(\bigoplus^n \sigma(c)) \| \geq \eps$$
for any unitary $u\in \mathrm{M}_{n+1}(B)$.

Chose sequences
\begin{enumerate}
\item $\mathcal G_1\subset \mathcal G_{2}\subset\cdots $ with $\overline{\bigcup_{k=1}^\infty \mathcal G_k} =A$,
\item $\mathcal H_1\subset \mathcal H_{2}\subset\cdots $ with $\overline{\bigcup_{k=1}^\infty \mathcal H_k} =A^+$,
\item $\delta_1>\delta_2>\cdots$ with $\lim_{k\to\infty}\delta_k=0$,
\end{enumerate}

For any $a\in A^+$ and any $\eps>0$, choose $a'\in \bigcup\mathcal H_k$ 
with $$\|a-a'\| < \frac{\eps}{L(a, \eps) N^2(a, \eps)},$$
and set new functions
$$\bar{L}(a, \eps)=L(a', \eps)\quad\mathrm{and}\quad \bar{N}(a, \eps)=N(a', \eps) $$

Applying Proposition \ref{no-1-uniq-hom-1} to $(\mathcal F, \eps)$ with respect to $\bar{L}(a, 2\eps)$ and $\bar{N}(a, 2\eps)$, one obtains $n$. Now, fix $(\mathcal F, \eps)$ and $n$. There are C*-algebras $B_k$ and $\mathcal G_k$-$\delta_k$-multiplicative maps $\phi_k, \psi_k, \sigma_k: A\to B_k$ satisfying 
$$\textrm{$\sigma(a)$ is $L(a, \eps)$-$N(a, \eps)$-full for any $\eps>0$ and any $a \in \mathcal H_k$,}$$
but
\begin{equation}\label{assumption}
\max_{c\in\mathcal F}\|u\mathrm{diag}\{\phi_k(c), \underbrace{\sigma_k(c), ..., \sigma_k(c)}_n\}u^* - \mathrm{diag}\{\psi_k(c), \underbrace{\sigma_k(c), ..., \sigma_k(c)}_n\} \| \geq \eps
\end{equation}
for any unitary $u\in {\mathrm{M}_{n+1}(B_k)}+\Comp 1_{n+1}$.

Consider the homomorphisms $\Phi, \Psi, \Sigma: A \to \prod B_k/\bigoplus B_k$ defined by $\prod_{k=1}^\infty \phi_k$,  $\prod_{k=1}^\infty \psi_k$,   and $\prod_{k=1}^\infty \sigma_k$, respectively. 

Then $\Sigma$ is $\bar{L}(a, 2\eps)$-$\bar{N}(a, 2\eps)$-full. Indeed, let $a\in \bigcup \mathcal H_k$, and assume that $a\in\mathcal H_1$.  For any $\eps>0$, and any $b=(b_1, b_2, ...)\in \prod B_k$ with $\|b\| \leq 1$, since $\sigma_k(a), k=1, 2, ...$, are $L(a, \eps)$-$N(a, \eps)$-full, there are $x_{k, 1}, x_{k, 2}, ..., x_{k, L(a, \eps)}$ with $\|x_{k, i}\| \leq N(a, \eps)$ such that
$$ \|b_i-(x_{k, 1}\sigma_k(a)x^*_{k, 1} + \cdots + x_{k, L(a, \eps)}\sigma_k(a)x^*_{k, L(a, \eps)}) \| <\eps.$$ Set $x_i=(x_{i, 1}, x_{i, 2}, ... )$,  Then $$\|x_i\| < N(a, \eps),\quad i=1, 2, ..., L(a, \eps)$$ 
and
$$ \|b-(x_1\Sigma(a)x^*_1 + \cdots + x_{L(a, \eps)}\Sigma(a)x^*_{L(a, \eps)}) \| <\eps.$$
Now, let $a\in A=\overline{\bigcup \mathcal H_k}$. Take the element $a'\in \bigcup \mathcal H_k$ with $$\|a-a'\| < \frac{\eps}{L(a, \eps) N^2(a, \eps)}.$$
Since $a'\in \bigcup \mathcal H_k$, there are $x_1, x_2, ..., x_{L(a', \eps)}$ with $\|x_i\| \leq N(a', \eps)=\bar{N}(a, \eps)$ such that 
$$ \|b-(x_1\Sigma(a')x^*_1 + \cdots + x_{L(a', \eps)}\Sigma(a')x^*_{L(a', \eps)}) \| <\eps,$$
and hence (note that $\bar{L}(a,\eps)=L(a', \eps)$)
$$ \|b-(x_1\Sigma(a)x^*_1 + \cdots + x_{L(a, \eps)}\Sigma(a')x^*_{\bar{L}(a, \eps)}) \| <2\eps.$$ 
Then, by Proposition \ref{no-1-uniq-hom-1}, there is a unitary $U\in {\mathrm{M}_{n+1}(B)}+\Comp 1_{n+1}$ such that
$$ \|U(\Phi(a)\oplus \underbrace{\Sigma(a) \oplus \cdots \oplus \Sigma(a)}_n)U^* - \Psi(a)\oplus \underbrace{\Sigma(a) \oplus \cdots \oplus \Sigma(a)}_n\| < \eps,\quad a\in\mathcal F. $$
By lifting $U$ to a unitary $(u_n) \in \mathrm M_{n+1}(\prod B_k) + \Comp 1_{n+1}$, one obtains a contradiction.
\end{proof}

\begin{rem}
If $A$ is simple, then the set $\mathcal H$ can be chosen to be a set consisting of a single point.
\end{rem}

In the case that the algebra has strict comparison of positive elements, the fullness of an element can be controlled by traces (Lemma \ref{full-by-trace}).

\begin{lem}\label{std-lem}
Let $A$ be a C*-algebra, and let $z\in A$. Set $a=zz^*$ and $b=z^*z$. Then, for any $\eps>0$, there is $u\in A$ such that
$\|u^*au-b\| < \eps$ and $\|u\|\leq 1$.
\end{lem}
\begin{proof}
Consider the polar decomposition $z=v(z^*z)^{\frac{1}{2}}$. Then $v^*av = b$. Set $u=vh_{\eps'}(z^*z)\in A$, where
$$
h_{\eps'}(t)=
\left\{
\begin{array}{ll}
1, & \textrm{if $\eps'\leq t\leq 1$}, \\
\textrm{linear}, & \textrm{if $\frac{\eps'}{2}\leq t\leq \eps'$}, \\
0, & \textrm{if $0\leq t\leq \frac{\eps'}{2}$}.
\end{array}
\right.
$$
With $\eps'$ sufficiently small, the element $u$ has the desired property.
\end{proof}

\begin{lem}\label{full-by-trace}
For any $\eps>0$ and $\delta>0$, there are $L(\delta)$ and $N(\eps)$ satisfying the following:
Let $A$ be a C*-algebra such that all traces of $A$ are bounded, and assume that $A$ has strict comparison of positive elements. Let $a\in A^+$ with $\|a\|=1$ satisfying
$$\tau(a) > \delta,\quad \tau\in\mathrm{T}^+(A),\ \|\tau\|=1.$$
Then, for any $b\in A^+$ with $\|b\|=1$, there are $x_1, x_2, ..., x_{L(\delta)}$ such that 
$$\|x_i\|\leq N(\eps),\quad i=1, 2, ..., L(\delta)$$
and
$$\|b-(x_1ax_1^* + x_2ax_2^*+\cdots+x_{L(\delta)}a x^*_{L(\delta)})\| < \eps.$$
\end{lem}

\begin{proof}
For each $\delta$, set $$L(\delta) = \lceil\frac{1}{\delta}\rceil.$$
For each $\eps>0$, 
define  $$
h_{\eps}(t)=
\left\{
\begin{array}{ll}
1, & \textrm{if $\eps\leq t\leq 1$}, \\
\textrm{linear}, & \textrm{if $\frac{\eps}{2}\leq t\leq \eps$}, \\
0, & \textrm{if $0\leq t\leq \frac{\eps}{2}$},
\end{array}
\right.
$$
and define
$$g_\eps(t)=\frac{h_\eps(t)}{t},\quad t\in (0, 1], \quad\textrm{and}\quad g_\eps(0)=0.$$
Then set
$$N(\eps)=(1+\frac{\eps}{5})\sqrt{\|g_\frac{\eps}{5}\|}.$$

Then $L(\delta)$ and $N(\eps)$ has the desired property.

Indeed, let $A$ be a C*-algebra with strict comparison of positive elements. Let $a\in A^+$ with $\|a\|=1$ satisfying
$$\tau(a) > \delta,\quad \tau\in\mathrm{T}^+(A),\ \|\tau\|=1.$$ It is then clear that
$$\mathrm{d}_\tau([a]) > \delta,\quad  \tau\in\mathrm{T}^+(A),\ \|\tau\|=1.$$

Let $b\in A^+$ with $\|b\| \leq 1$. Then $\mathrm{d}_\tau([b]) < \lceil\frac{1}{\delta}\rceil \mathrm{d}_\tau([a])$, $\tau\in\mathrm{T}(A)$. Since $A$ has strict comparison, one has
$$b \preceq \underbrace{a\oplus\cdots\oplus a}_{\lceil\frac{1}{\delta}\rceil=L(\delta)}=a^{\oplus L(\delta)}.$$

There is then an element $z\in \mathrm{M}_{L(\delta)}(A)$ such that
$$(b-\frac{\eps}{5})_+=z^*z\quad\mathrm{and}\quad zz^*\in \mathrm{M}_{L(\delta)}(\overline{aAa}).$$
Pick $x\in \mathrm{M}_{L(\delta)}(A)$ such that
\begin{equation}\label{est-913-01}
\|zz^* - (a^{\oplus L(\delta)})xx^*(a^{\oplus L(\delta)})\| < \frac{\eps}{5}.
\end{equation}

Moreover, note that $x$ can be chosen so that $$\|(a^{\oplus L(\delta)})x\|\in [1-\frac{2\eps}{5}, 1],$$ hence
\begin{equation}\label{est-913-02}
\|x^*(a^{\oplus L(\delta)}) h_{\frac{\eps}{5}}(a^{\oplus L(\delta)})(a^{\oplus L(\delta)})x - x^*(a^{\oplus L(\delta)}) (a^{\oplus L(\delta)})x\| < \frac{\eps}{5}.
\end{equation}

By Lemma \ref{std-lem}, there is $u_1\in \mathrm{M}_{L(\delta)}(A)$ such that $\|u_1\| < \sqrt{1+\eps}$ and
\begin{equation}\label{est-913-03}
\|u_1(x^*(a^{\oplus L(\delta)}) (a^{\oplus L(\delta)})x)u_1^* - (a^{\oplus L(\delta)})xx^*(a^{\oplus L(\delta)})\| < \frac{\eps}{5}.
\end{equation}

Again by Lemma \ref{std-lem}, there is $u_2\in \mathrm{M}_{L(\delta)}(A)$ such that $\|u_2\| < \sqrt{1+\eps}$ and
\begin{equation}\label{est-913-04}
\|u_2(zz^*)u^*_2 - (b-\frac{\eps}{5})_+\|< \frac{\eps}{5},
\end{equation}
and hence (by \eqref{est-913-01}, \eqref{est-913-02}, \eqref{est-913-03}, and \eqref{est-913-04})
$$\|u_2u_1(x^*(a^{\oplus L(\delta)}) h_{\frac{\eps}{5}}(a^{\oplus L(\delta)})(a^{\oplus L(\delta)})x) u_1^*u_2^* - b \|< \eps,$$
In other words, there is an element $w'\in \mathrm{M}_{L(\delta)}(A)$ with $\|w'\|\leq 1+\eps$ such that
$$\|w'(h_{\frac{\eps}{5}}(a^{\oplus L(\delta)}))(w')^* - b \| < \eps.$$
Note that $h_{\frac{\eps}{5}}(a) = g_{\frac{\eps}{5}}(a) a$. One has that there is $w\in \mathrm{M}_{L(\delta)}(A)$ with $\|w\|\leq (1+\eps)\sqrt{\|g_{\frac{\eps}{5}}\|}$ such that
$$\|w(\bigoplus^La)w^* - b \| < \eps.$$
Write $w=(w_{i,j})$, $i, j=1, 2, ..., L$. Then
$$\| b - \sum_{i=1}^L w_{1, i} a w_{1, i}^*\| < \eps.$$
Since $\|w\|\leq (1+\eps)\sqrt{\|g_{\frac{\eps}{5}}\|}$, one has that 
$$\|w_{i, j}\|\leq (1+\eps)\sqrt{\|g_\frac{\eps}{5}\|}=N(\eps),\quad i, j=1, 2, ..., L,$$
as desired.
\end{proof}

\begin{cor}\label{stable-uniq-trace}
Let $A$ be a separable nuclear C*-algebra without a unit which is KK-contracible. Let $\rho: A^+\setminus\{0\} \to (0, +\infty)$ be a map.

Let $\mathcal F\subset A$ be a finite set, and let $\eps>0$.  There exist $\mathcal G\subset A$, $\mathcal H\subset A^+$, $\delta>0$, $n\in\mathbb N$ such that for any $\mathcal G$-$\delta$-multiplicative maps $\phi, \psi, \sigma: A \to B$, where $B$ is a C*-algebra with all traces are bounded and with strict comparison of positive elements, such that
$$\tau(\sigma(a)) > \rho(a),\quad a\in\mathcal H,\ \tau\in\mathrm{T}^+B,\ \|\tau\|=1,$$
then there is a unitary $u\in\mathrm{M}_{n+1}(B) + \Comp 1_{n+1}$ such that
$$\| u(\phi(a)\oplus\underbrace{\sigma(a)\oplus\cdots\oplus\sigma(a)}_n)u^* - \psi(a)\oplus\underbrace{\sigma(a)\oplus\cdots\oplus\sigma(a)}_n\| <\eps,\quad a\in\mathcal F.$$
\end{cor}
\begin{proof}
The statement follows directly from Lemma \ref{full-by-trace} and Proposition \ref{stable-uniq-full}.
\end{proof}

A special case of the stable uniqueness theorem is when $A=B$, $\phi=\sigma=\mathrm{id}$, and $\psi=0$. 
\begin{cor}\label{zero-head}
Let $A$ be a separable simple nuclear C*-algebra without a unit which is KK-equivalent to $\{0\}$. Assume all traces of $A$ are bounded and $A$ has strict comparison of positive elements. 

Let $\mathcal F\subset A$ be a finite set, and let $\eps>0$. There is $n\in\mathbb N$ and a unitary $u\in\mathrm{M}_{n+1}(A) + \Comp 1_{n+1}$ such that
$$\| u(a\oplus\underbrace{ a \oplus\cdots\oplus a }_n)u^* - 0\oplus\underbrace{a\oplus\cdots\oplus a}_n\| <\eps,\quad a\in\mathcal F.$$
\end{cor}
\begin{proof}
For any $a\in A^+\setminus\{0\}$, define $\rho(a)=\inf\{\frac{1}{2}\tau(a): \tau\in\mathrm{T}^+A,\ \|\tau\|=1\}$. Since $\mathrm{T}^+A$ has a compact base, and $A$ is simple, one has that $\rho(a) \in (0, \infty)$. Then, the statement follows from Corollary \ref{stable-uniq-trace}.
\end{proof}

\begin{cor}\label{UHF-stable}
Let $A$ be a separable simple nuclear C*-algebra without a unit which is KK-contractible. Assume that all traces of $A$ are bounded. Then $A\otimes U \cong A\otimes Q$ for any UHF algebra $U$.
\end{cor}
\begin{proof}
The statement follows from Corollary \ref{zero-head} and an approximate intertwining argument. We only sketch the outline in the following, and leave the details to readers.

For any UHF algebra $U$, there are UHF algebras $U_1, U_2$ such that $U\cong U_1\otimes U_2$. Since $Q\cong U_1 \otimes Q$, in order to prove the statement, one may replace $A$ by $A\otimes U_1$, and thus one may assume that $A$ has strict comparison on positive elements (Corollary 4.7 of \cite{Ror-Z-stable}).

Write
$$
\xymatrix{
 \mathrm{M}_{n_1}(A) \ar[r] & \mathrm{M}_{n_2}(A) \ar[r] & \cdots \ar[r] & A\otimes U \\
 \mathrm{M}_{m_1}(A) \ar[r] & \mathrm{M}_{m_2}(A) \ar[r] & \cdots \ar[r] & A\otimes Q ,
}
$$
where the connecting maps have the form $a\to a\otimes 1_{n_k}$.


On applying Corollary \ref{zero-head}, after a telescoping operation in each inductive limit, there is an approximate commutative diagram
$$
\xymatrix{
 \mathrm{M}_{n_1}(A) \ar[r]^-{\mathrm{ad}u_1} \ar[d]_{\phi_1} & \mathrm{M}_{n_2}(A) \ar[d]_{\phi_2} \ar[r]^-{\mathrm{ad}u_2} & \cdots \ar[r] & A\otimes U \\
 \mathrm{M}_{m_1}(A) \ar[r]_-{\mathrm{ad}w_1} \ar[ur]_{\psi_1} & \mathrm{M}_{m_2}(A) \ar[r]_-{\mathrm{ad}w_2} \ar[ur]_{\psi_2} & \cdots \ar[r] & A\otimes Q ,
}
$$
where each of the maps $\phi_i$ and $\psi_j$ is $a \mapsto 0 \oplus \underbrace{a \oplus a\oplus \cdots \oplus a}_k$ for a sufficiently large $k$, and therefore $A\otimes U \cong A\otimes Q$.
\end{proof}

\begin{lem}\label{untwrist}
Let $A$ be a  separable simple nuclear C*-algebra which is KK-contractible, and assume that all traces of $A$ are bounded. Let $M_p$ and $M_q$ be two UHF algebras. Then, there exist an isomorphism $\phi: A\otimes \mathrm{M}_p \to A \otimes \mathrm{M}_p \otimes\mathrm{M}_q$ and a continuous path of unitaries $u_t\in \mathcal M(A\otimes\mathrm{M}_p\otimes \mathrm{M}_q)$, $1\leq t < \infty$, such that $u_1=1$ and
$$ \lim_{t\to\infty} u^*_t(a \otimes r\otimes 1_q)u_t = \phi(a\otimes r),\quad a\in A,\ r\in \mathrm{M}_p.$$
\end{lem}

\begin{proof}
It follows from Lemma \ref{UHF-stable} that $A\otimes \mathrm{M}_p\cong A\otimes Q$. Thus, it is enough to show that 
there exist an isomorphism $\phi: A\otimes Q \to A \otimes Q \otimes\mathrm{M}_q$ and a continuous path of unitaries $u_t\in \mathcal M(A\otimes Q \otimes \mathrm{M}_q)$, such that $u_1=1$ and
$$ \lim_{t\to\infty} u^*_t(a \otimes r\otimes 1_q)u_t = \phi(a\otimes r),\quad a\in A,\ r\in Q.$$

Pick any isomorphism $\psi: Q \to Q \otimes\mathrm{M}_q$. It is clear that $\phi:=\mathrm{id}_A\otimes \psi: A\otimes Q \to A\otimes Q \otimes \mathrm{M}_q$ is an isomorphism. Then, the embedding map $Q\ni r\mapsto r\otimes 1_q\in Q\otimes\mathrm{M}_q$ and the isomorphism $\psi$ are asymptotically unitarily equivalent; that is, there is a continuous part of unitary $v_t\in Q\otimes\mathrm{M}_q$, $1\leq t <\infty$, such that 
$$\lim_{t\to\infty}v_t^*(r\otimes 1_p)v_t = \psi(r),\quad r\in Q.$$
One may assume that $v_1=1_{Q\otimes\mathrm{M}_q}$.

Consider the path
$$u_t = 1_{\tilde{A}}\otimes v_t \in \mathcal M(A\otimes Q \otimes \mathrm{M}_q).$$
Then, for any $a\otimes r\in A\otimes Q$, 
\begin{eqnarray*} 
\lim_{t\to\infty} u^*_t(a \otimes r\otimes 1_q)u_t & = &  \lim_{t\to\infty} (1_{\tilde{A}}\otimes v_t^*)(a \otimes r\otimes 1_q)(1_{\tilde{A}}\otimes v_t) \\
& = & \lim_{t\to\infty} a \otimes (v_t^*(r\otimes 1_q)v_t )\\
& = & a\otimes \psi(r) \\
& = & \phi(a\otimes r),
\end{eqnarray*}
as desired.
\end{proof}

\begin{thm}\label{Z-stable}
Let $A$ be a separable simple nuclear C*-algebra without a unit which is KK-contractible. Then $A\otimes \mathcal Z \cong A\otimes Q$.
\end{thm}

\begin{proof}
To prove the statement, it is enough to assume that $A$ is $\mathcal Z$-stable and to show $A$ is $Q$-stable, i.e., $A\cong A\otimes Q$.

Let us first prove the statement for $A$ such that all traces of $A$ are bounded. Write $A\otimes \mathcal Z$ as an inductive limit of $A\otimes\mathcal Z_{p, q}$, where $p, q$ are two relatively prime supernatural numbers such that $\mathrm{M}_p\otimes\mathrm{M}_q = Q$. By corollary 3.4 of \cite{TW-D}, in order to show that $A$ is $Q$-stable, it is enough to show that $A\otimes\mathcal Z_{p, q}$ is $Q$-stable. Note that
$$A\otimes\mathcal Z_{p, q} =\{f\in\mathrm{C}([0, 1],  A\otimes\mathrm{M}_p\otimes\mathrm{M}_q): f(0)\in A\otimes \mathrm{M}_p\otimes 1_q,\ f(1)\in A\otimes 1_p\otimes \mathrm{M}_q \}.$$ Applying Lemma \ref{untwrist} to both endpoints, one obtains isomorphisms $$\phi_0: A\otimes \mathrm{M}_p \to A \otimes \mathrm{M}_p \otimes\mathrm{M}_q,\quad \phi_1: A\otimes \mathrm{M}_q \to A \otimes \mathrm{M}_p \otimes\mathrm{M}_q,$$
together with a continuous path of unitaries $u_t\in \mathcal M(A\otimes\mathrm{M}_p\otimes \mathrm{M}_q)$, $0 < t < 1$, such that $u_{\frac{1}{2}}=1$, 
$$ \lim_{t\to 0} u^*_t(a \otimes r\otimes 1_q)u_t = \phi_0(a\otimes r),\quad a\in A,\ r\in \mathrm{M}_p,$$
and
$$ \lim_{t\to 1} u^*_t(a \otimes 1_p \otimes r)u_t = \phi_1(a\otimes r),\quad a\in A,\ r\in \mathrm{M}_q.$$

Define the continuous field map
$$\Phi: A \otimes \mathcal Z_{p, q} \to \mathrm{C}([0, 1], A\otimes \mathrm{M}_p\otimes\mathrm{M}_q)$$ by
$$\Phi(f)(t) = u_t^*f(t)u_t,\quad t\in[0, 1],$$
where $\Phi(f)(0)$ and $\Phi(f)(1)$ are understood as $\phi_0(f(0))$ and $\phi_1(f(1))$, respectively. Then the map $\Phi$ is an isomorphism (the inverse  is $\Phi^{-1}(g)(t) = u_tg(t)u^*_t$, $t\in(0, 1)$, $\Phi^{-1}(g)(0)=\phi_0^{-1}(g(0))$, and $\Phi^{-1}(g)(1)=\phi_0^{-1}(g(1))$), and hence $A\otimes\mathcal Z_{p, q} \cong \mathrm{C}([0, 1], A\otimes \mathrm{M}_p\otimes\mathrm{M}_q).$ Since the trivial field $\mathrm{C}([0, 1], A\otimes \mathrm{M}_p\otimes\mathrm{M}_q)$ is $Q$-stable, one has that $A\otimes\mathcal Z_{p, q}$ is $Q$-stable, as desired.

For a general simple C*-algebra $A$, pick a hereditary sub-C*-algebra $B\subset A$ such that all traces of $B$ are bounded. By the Brown's theorem, the C*-algebra $A$ and the C*-algebra $B$ are stably isomorphic, and hence $B$ is also KK-contractible. Since $A$ is $\mathcal Z$-stable, it follows from Corollary 3.2 of \cite{TW-D} that $B$ is $\mathcal Z$-stable. Then, by the argument above, one has that $B$ is $Q$-stable, and hence $A\otimes\mathcal K \cong B\otimes\mathcal K$ is $Q$-stable. Using Corollary 3.2 of \cite{TW-D} again, one has that $A$ is $Q$-stable, as desired.
\end{proof}

\subsection{Uniqueness theorem}


\begin{lem}[Theorem 6 of \cite{Robert-Cu}]\label{uniq-R}
Consider $C\in \mathcal R_z$. For any finite set $\mathcal F\subset C$ and any $\eps>0$, there are finite set $\mathcal G\subset \mathrm{Cu}(C)$ such that for any C*-algebra $A$ with stable rank one and any homomorphisms $\phi, \psi: C\to A$, if
$$\phi(g') \leq \psi(g)\quad\textrm{and}\quad \psi(g') \leq \textrm{$\phi(g),\quad g', g\in\mathcal G$ with $g'\ll g$},$$
then there is a unitary $u\in \tilde{A}$ such that
$$\|u^*\phi(a) u -\psi(a)\|<\eps,\quad a\in \mathcal F.$$
\end{lem}

\begin{cor}\label{uniq-R-cor}
Consider $C\in \mathcal R_z$. Let $\rho: C^+\setminus\{0\}\to (0, +\infty)$ be a map. For any finite set $\mathcal F\subset C$ and any $\eps>0$, there are finite set $\mathcal H_0, \mathcal H_1 \subset C^+$ and $\sigma>0$ such that for any C*-algebra $A$ with all traces bounded and $A$ has strict comparison of positive elements, and any homomorphisms $\phi, \psi: C\to A$, if
\begin{enumerate}
\item $\tau(\phi(a)), \tau(\psi(a)) > \rho(a)$, $a\in\mathcal H_0$, $\tau\in\mathrm{T}^+(A)$, $\|\tau\| =1$, and 
\item $|\tau(\phi(a) - \psi(a))| < \sigma$, $a\in\mathcal H_1$, $\tau\in\mathrm{T}^+(A)$, $\|\tau\| =1$, 
\end{enumerate}
then there is a unitary $u\in \tilde{A}$ such that
$$\|u^*\phi(a) u -\psi(a)\|<\eps,\quad a\in \mathcal F.$$
\end{cor}
\begin{proof}
Let $\mathcal G\subset \mathrm{Cu}(C)$ be the finite subset obtained from Lemma \ref{uniq-R}. Write $$\mathcal G=\{g_1, g_2, ..., g_n\}$$ and list $\{(h'_i, h_i): i=1, 2, ..., m\}$ as the list of pairs with $h'_i \ll h_i$. Note that since $C$ is stably projectionless, none of $g_i$ are from projections. Then, pick elements $c_i', c_i \in \mathrm{M}_\infty(C)$, such that
$$c_i' < c_i\quad \mathrm{and}\quad h_i'<[c'_i] < [c_i] < h_i,$$
and $$\mathrm{d}_\tau(h_i') \leq \tau (c'_i)\leq \tau(c_i) \leq d_\tau(h_i),\quad \tau\in\mathrm{T}^+ C.$$

For each positive element $c_i-c_i'$, pike a positive element $d_i\in C$ such that $$0<d_i < c_i-c_i'.$$
Then set
$$\mathcal H_0=\{d_1, d_2, ..., d_m\}.$$
and set $$\sigma'=\frac{1}{2}\min\{\rho(d_i): i=1, 2, ..., m\}$$
Also choose a finite set $\mathcal H\subset C$ and $\sigma>0$ such that if $\tau_1, \tau_2$ are linear functional of $C$ with norm at most one satisfying $$|\tau_1(a) - \tau_2(a)| < \sigma,\quad a\in\mathcal H_1,$$ then
$$|\tau_1(b) - \tau_2(b)| < \sigma',\quad b=c'_i, c_i,\ i=1, 2, ..., m.$$

Then, the tuple $(\mathcal H_0, \mathcal H_1, \sigma)$ satisfies the conclusion of the lemma. Indeed, let $A$ and $\phi, \psi: C\to A$ be given satisfying the assumption of the lemma. Then, for any $\tau\in\mathrm{T}^+A$ with $\|\tau\| = 1$, one has
$$\tau(\phi(c_i)) - \tau(\phi(c'_i)) \geq \tau (\phi(d_i)) > \rho(d_i)>2\sigma'$$
$$\tau(\psi(c_i)) -\tau(\psi(c'_i)) \geq \tau (\psi(d_i)) > \rho(d_i)>2\sigma'$$
and
$$|\tau(\phi(c_i')) -\tau(\psi(c_i'))| < \sigma'\quad\mathrm{and}\quad |\tau(\phi(c_i)) -\tau(\psi(c_i))| < \sigma'.$$
Therefore
$$\mathrm{d}_\tau(\psi(h_i')) \leq \tau(\psi(c'_i)) < \tau(\phi(c'_i)) + \sigma' < \tau(\phi(c_i))-\sigma' < \tau(\phi(c_i))\leq \mathrm{d}_\tau(\phi(h_i))$$ and
$$\mathrm{d}_\tau(\phi(h_i'))\leq\tau(\phi(c'_i)) < \tau(\psi(c'_i)) + \sigma' < \tau(\psi(c_i))-\sigma' < \tau(\psi(c_i))\leq \mathrm{d}_\tau(\psi(h_i)).$$
Since $A$ has strict comparison, one has
$$\psi(h'_i) < \phi(h_i)\quad\mathrm{and}\quad \phi(h'_i) < \psi(h_i),\quad i=1, 2, ...,$$
and therefore there is a unitary $u\in\tilde{A}$ such that
$$\|u^*\phi(a) u -\psi(a)\|<\eps,\quad a\in \mathcal F$$
as desired.
\end{proof}

\begin{thm}\label{uniq-thm-0}
Let $A\in\mathcal R$ be a simple separable C*-algebra, and let $\rho: A^+\setminus\{0\} \to (0, +\infty)$ be a function satisfying
$$\tau(a) \geq \rho(a),\quad \tau\in\mathrm{T}^+A,\ \|\tau\|=1.$$ Also assume that $A$ has finite nuclear dimension and $A$ is KK-equivalent to $\{0\}$. 

Then, for any finite set $\mathcal F\subset A\otimes 1_Q$ and any $\eps>0$, there are finite sets $\mathcal G\subset A, \mathcal H_0, \mathcal H_1\subset (A\otimes Q)^+$ and $\delta, \sigma>0$ such that if $\phi: A\to A$ is a $\mathcal G$-$\delta$-multiplicative linear contraction such that
\begin{enumerate}
\item $\tau((\phi\otimes\mathrm{id})(a)) \geq \rho(a)$, $a\in\mathcal H_0$, $\tau\in\mathrm{T}^+(A\otimes Q)$, $\|\tau\|=1$,
\item $|\tau((\phi\otimes\mathrm{id})(a)) - \tau(a)| < \sigma$, $a\in\mathcal H_1$, $\tau\in\mathrm{T}^+(A\otimes Q)$, $\|\tau\|=1$,
\end{enumerate}
then there is a unitary $u\in \tilde{A}\otimes Q$ such that
$$\|u^*(\phi(a)\otimes 1_Q) u - a\otimes 1_Q \| < \eps,\quad a\in\mathcal F.$$
\end{thm}

\begin{proof}
Applying Corollary \ref{stable-uniq-trace} to $\mathcal F$ and $\eps/2$ with respect to $A$ and $\frac{1}{2}\rho$, one obtains $(\mathcal G_0, \delta_0)$, $\mathcal H$, and $n$.

Applying Theorem \ref{div-decp}, for arbitrary $\eps'>0$, there are projections $q, p_1, p_2, ..., p_n \in \tilde{A}\otimes Q$ and $S_1\subset p_1(A\otimes Q)p_1$ with $S_1\in\mathcal R_0$ such that
\begin{enumerate}
\item $q+p_1+p_2+\cdots+p_n =1_{\tilde{A}\otimes Q}$, $q$ is Murray-von Neumann equivalent to $p_i$, $i=1, 2, ..., n$,
\item $\|p_ia-ap_i\| < \eps'$, $a\in\mathcal F\cup \mathcal G \cup\mathcal H$,
\item $p_iap_i\in_{\eps'} S_i:=v_iSv^*_i$, $a\in\mathcal F\cup \mathcal G \cup\mathcal H$, where $v_i\in \tilde{A}\otimes Q$, $i=1, 2, ..., n$, are the partial isometries  satisfying $v^*_iv_i=p_1$ and $v_iv_i^*=p_i$.
\end{enumerate}
One then choose $\eps'$ sufficiently small so that
\begin{equation}\label{decom-10-01}
\|a-(qaq + p_1ap_1+p_2ap_2 + \cdots + p_nap_n)\| < \frac{\eps}{8},\quad a\in \mathcal F,
\end{equation}
the maps $a\mapsto qaq$ and $a\mapsto p_iap_i$, $i=1, 2, ..., n$, are $\mathcal G_0$-$\delta_0$-multiplicative,
and
\begin{equation}\label{tr-lbd-10}
\tau(p_iap_i) > \frac{1}{2}\rho(a),\quad a\in\mathcal H,\ \tau\in\mathrm{T}^+(p_i(A\otimes Q)p_i),\ \|\tau\|=1.
\end{equation}

Applying Corollary \ref{uniq-R-cor} to $p_i\mathcal Fp_i$ and $\frac{\eps}{8}>0$ (again, with $\eps'$ sufficiently small, let us assume that $p_i\mathcal Fp_i\subset S_i$) 
 with respect to $\frac{1}{2}\rho$ and $S_i$, one obtains $\mathcal H_{0, i}, \mathcal H_{1, i}\subset S_i\subset A\otimes Q$ and $\sigma_0>0$. Set $\mathcal H_0=\bigcup_{i=1}^n\mathcal H_{0, i}$ and $\mathcal H_1=\bigcup_{i=1}^n\mathcal H_{1, i}$.

Choose a finite set $\mathcal G\subset A$ sufficiently large and $\delta>0$ sufficiently small that if $\phi: A \to A$ is $\mathcal G$-$\delta$-multiplicative, then there is $\tilde{\Phi}: \tilde{A}\otimes Q \to \tilde{A}\otimes Q$ such that  the restriction of $\tilde{\Phi}$ to the $(n+1)\times(n+1)$ matrix algebra generated by $q, p_1, p_2, ..., p_n$ is a homomorphism, and
$$\|(\tilde{\phi}\otimes \mathrm{id})(a)-\tilde{\Phi}(a) \| < \eps''',\quad a\in \mathcal H_0\cup\mathcal H_1\cup \bigcup_{i=1}^n p_i\mathcal Fp_i,$$
where $$\eps'''=\min\{\frac{1}{24}\eps, \frac{1}{4}\sigma_0, \frac{1}{4}\rho(h): h\in\mathcal H_0\}.$$
Moreover, the map $a\mapsto \tilde{\Phi}(q(a\otimes 1)q)$ is $\mathcal G_0$-$\delta_0$-multiplicative.


Then $(\mathcal G, \delta)$, $\mathcal H_0$, $\mathcal H_1$, and $\sigma=\frac{1}{2}\sigma_0$ satisfy the theorem.

Indeed, if $\phi: A \to A$ is a $\mathcal G$-$\delta$-multiplicative contractive map satisfying the conditions of the theorem, then 
\begin{equation}\label{tr-cond-10-01}
\tau((\phi\otimes\mathrm{id})(a)) \geq \rho(a),\quad a\in\mathcal H_0,\ \tau\in\mathrm{T}^+(A\otimes Q),\ \|\tau\|=1,
\end{equation}
and
\begin{equation}\label{tr-cond-10-02}
\|\tau(\phi(a)) - \tau(a)\| < \sigma,\quad a\in\mathcal H_1,\ \tau\in\mathrm{T}^+(A\otimes Q),\ \|\tau\|=1. 
\end{equation}

Denote by $\tilde{\phi}: \tilde{A} \to \tilde{A}$ the unitization of $\phi$, and consider the map $\tilde{\phi}\otimes \mathrm{id}: \tilde{A}\otimes Q \to \tilde{A}\otimes Q$. Since $\phi$ is $\mathcal G$-$\delta$-multiplicative, 
there is $\tilde{\Phi}: \tilde{A}\otimes Q \to \tilde{A}\otimes Q$ such that  the restriction of $\tilde{\Phi}$ to the $(n+1)\times(n+1)$ matrix algebra generated by $q, p_1, p_2, ..., p_n$ is a homomorphism, and
\begin{equation}\label{pert-10-03}
\|(\tilde{\phi}\otimes \mathrm{id})(a)-\tilde{\Phi}(a) \| < \eps''',\quad a\in \mathcal H_0\cup\mathcal H_1\cup \bigcup_{i=1}^n p_i\mathcal F p_i.
\end{equation}
Moreover, the map $a\mapsto \tilde{\Phi}(q(a\otimes 1)q)$ is $\mathcal G_0$-$\delta_0$-multiplicative.
Note that $$[\tilde{\Phi}(q)]=[(\tilde{\phi}\otimes\mathrm{id})(q)]=[q]\quad\mathrm{and}\quad [\tilde{\Phi}(p_i)]=[(\tilde{\phi}\otimes\mathrm{id})(p_i)]=[p_i],\quad i=1, 2, ..., n.$$ 
Since $\tilde{A}\otimes Q$ has stable rank one, there is a unitary $u_0\in \tilde{A}\otimes Q$ such that 
$$u_0^*\tilde{\Phi}(q)u_0=q\quad\mathrm{and}\quad u_0^*\tilde{\Phi}(p_i)u_0=p_i,\quad i=1, 2, ..., n,$$
and hence the image of $p_i(A\otimes Q)p_i$ (or $q(A\otimes Q)q$) under the map $u_0^*\tilde{\Phi}u_0$ is inside $p_i(A\otimes Q)p_i$ (or $q(A\otimes Q)q$). 

Note that by \eqref{pert-10-03}, for any $a\in\mathcal H_{0, i}\cup\mathcal H_{1, i} \cup p_i\mathcal Fp_i$,
\begin{equation}\label{pert-10-04}
p_i(u_0^*(\phi\otimes\mathrm{id})(a)u_0)p_i \approx_{\eps'''} p_i(u_0^*\tilde{\Phi}(a)u_0)p_i = u_0^*\tilde{\Phi}(a)u_0 \approx_{\eps'''} u_0^*(\phi\otimes\mathrm{id})(a)u_0.
\end{equation}
Consider the map
$$S_i\ni a \mapsto p_iu_0^*(\phi\otimes\mathrm{id})(a)u_0p_i\in p_i(A\otimes Q)p_i.$$
By \eqref{tr-cond-10-01}, \eqref{tr-cond-10-02}, and \eqref{pert-10-04},
$$\tau(p_iu_0^*(\phi\otimes\mathrm{id})(a)u_0p_i) \geq \frac{1}{2}\rho(a),\quad a\in\mathcal H_{0, i},\ \tau\in\mathrm{T}^+(p_i(A\otimes Q)p_i),\ \|\tau\|=1,$$
and
$$\|\tau(p_iu_0^*(\phi\otimes\mathrm{id})(a)u_0p_i) - \tau(a)\| < \sigma_0,\quad a\in\mathcal H_{1, i},\ \tau\in\mathrm{T}^+(p_i(A\otimes Q)p_i),\ \|\tau\|=1. $$
It then follows from Corollary \ref{uniq-R-cor} (with $\psi$ being the embedding map $S_i\subset p_iAp_i$) that there is a unitary $u_{1, i}\in p_i(\tilde{A}\otimes Q)p_i$ such that
$$\|u_{1, i}^*p_iu_0^*((\phi\otimes \mathrm{id})(a))u_0p_iu_{1, i} - a\| < \frac{\eps}{8},\quad a\in p_i\mathcal Fp_i,$$
and hence by \eqref{pert-10-04}, one has
\begin{equation}\label{pert-10-05}
\|u_{1, i}^*u_0^*((\phi\otimes \mathrm{id})(a))u_0u_{1, i} - a\| < \frac{\eps}{8}+2\eps''',\quad a\in p_i\mathcal Fp_i.
\end{equation}

Put $$u_1=q\oplus u_{1, i}\oplus\cdots\oplus u_{1, n}.$$ One has that for any $a\in \mathcal F$,
\begin{eqnarray*}
&& u_1^*u_0^*((\phi\otimes\mathrm{id})(a))u_0u_1 \\
&\approx_{\frac{\eps}{8}}& u_1^*u_0^*((\phi\otimes\mathrm{id})(q(a\otimes 1_Q)q+\sum_{i=1}^np_i(a\otimes 1_Q)p_i))u_0u_1 \quad\quad(\textrm{by \eqref{decom-10-01}}) \\
&=& u_1^*u_0^*((\phi\otimes\mathrm{id})(q(a\otimes 1_Q)q))u_0u_1+ \sum_{i=1}^nu_1^*u_0^*(\phi\otimes\mathrm{id}(p_i(a\otimes 1_Q)p_i))u_0u_1 \\
&\approx_{\frac{\eps}{8}+2\eps'''}& u_0^*((\phi\otimes\mathrm{id})(q(a\otimes 1_Q)q))u_0 + \sum_{i=1}^n p_i(a\otimes 1_Q)p_i\quad\quad (\textrm{by \eqref{pert-10-05}})\\
&\approx_{\eps'''}& u_0^*(\tilde{\Phi}(q(a\otimes 1_Q)q))u_0 + \sum_{i=1}^n p_i(a\otimes 1_Q)p_i, \quad\quad (\textrm{by \eqref{pert-10-03}}).
\end{eqnarray*}
Thus,
\begin{equation}\label{pert-10-06}
\|u_1^*u_0^*((\phi\otimes\mathrm{id})(a))u_0u_1 - (\phi_0(a)+p_1ap_1+\cdots + p_nap_n)\| < \frac{\eps}{4} +3\eps''',\quad a\in\mathcal F,
\end{equation}
where
$$\phi_0: A \ni a \mapsto u^*_0\tilde{\Phi}(q(a\otimes 1_Q)q)u_0  \in q(A\otimes Q)q.$$
Note that (\eqref{tr-lbd-10})
$$\tau(p_iap_i) \geq\frac{1}{2}\rho(a),\quad a\in\mathcal H,\ \tau\in \mathrm{T}^+(p_iAp_i), \ \|\tau\|=1,$$
and the maps $a\mapsto p_iap_i$, $a\mapsto qaq$, and $\phi_0$ are $\mathcal G_0$-$\delta_0$-multiplicative;
then, compare the maps
$$A \ni a\mapsto \phi_0(a) + p_1ap_1 +\cdots + p_nap_n \in A\otimes Q\cong \mathrm{M}_{n+1}(q(A\otimes Q)q)$$
and 
$$A \ni a\mapsto qaq + p_1ap_1 +\cdots + p_nap_n \in A\otimes Q \cong  \mathrm{M}_{n+1}(q(A\otimes Q)q).$$
By Corollary \ref{stable-uniq-trace}, there is a unitary $u_2\in \tilde{A}\otimes Q$
such that 
\begin{equation}\label{final-pert}
\|u^*_2(\phi_0(a) + p_1ap_1 +\cdots + p_nap_n)u_2 - (qaq + p_1ap_1 +\cdots + p_nap_n)\| < \eps/2,\quad a\in \mathcal F.
\end{equation}
Hence, by \eqref{final-pert}, \eqref{pert-10-06}, \eqref{decom-10-01}, one has
$$\|u_2^*u^*_1u^*_0((\phi(a)\otimes1_Q)u_0u_1u_2 - a\otimes 1_Q \| < \frac{7}{8}\eps+3\eps'''\leq \eps,\quad a\in \mathcal F,$$
as desired.
\end{proof}

Note that any isomorphism $Q\to Q\otimes Q$ is approximately unitarily equivalent to the embedding $Q\ni a\mapsto a\otimes 1\in Q\otimes Q$. Also note that for any simple C*-algebra in the reduction class, one always can define
$$\rho: A^+\ni a \mapsto \min\{\tau(a): \tau\in\mathrm{T}^+A,\ \|\tau\|=1\} \in(0, +\infty),$$
and it satisfies $$\tau(a) \geq \rho(a),\quad \tau\in\mathrm{T}^+A,\ \|\tau\|=1.$$ So, one may remove the density function $\rho$ from the statement of Theorem \ref{uniq-thm-0}:
\begin{thm}\label{uniq-thm}
Let $A\in\mathcal R$ be a simple separable C*-algebra. 
Assume that $A$ has finite nuclear dimension, $A$ is KK-equivalent to $\{0\}$, and $A\cong A\otimes Q$. 

Then, for any finite set $\mathcal F\subset A\otimes Q$ and any $\eps>0$, there are finite sets $\mathcal G\subset A\otimes Q, \mathcal H\subset (A\otimes Q)^+$ and $\delta, \sigma>0$ such that if $\phi: A\otimes Q \to A\otimes Q$ is a $\mathcal G$-$\delta$-multiplicative linear contraction such that
$$|\tau(\phi(a)) - \tau(a)| < \sigma,\ a\in\mathcal H,\ \tau\in\mathrm{T}^+(A\otimes Q),\ \|\tau\|=1,$$
then there is a unitary $u\in \tilde{A}\otimes Q$ such that
$$\|u^*\phi(a) u - a \| < \eps,\quad a\in\mathcal F.$$
\end{thm}

\subsection{Classification theorem}

\begin{thm}\label{clas-thm-R}
Let $A, B\in\mathcal R$ be simple separable KK-contractible C*-algebra with finite nuclear dimension.  
Suppose that 
$$\xi: (\mathrm{T}^+B, \Delta_B) \to (\mathrm{T}^+A, \Delta_A),$$
where $\Delta_B$ and $\Delta_A$ are simplexes of tracial states of $B$ and $A$, respectively. Then, one has that $A\cong B$, and the C*-algebra isomorphism can be chosen to carry the invariant isomorphism.
\end{thm}

\begin{proof}
Since $A$, $B$ have finite nuclear dimension, one has that $A\cong A\otimes\mathcal Z$ and $B\cong B\otimes\mathcal Z$. By Theorem \ref{Z-stable}, $A\cong A\otimes Q$ and $B\cong B\otimes Q$, where $Q$ is the universal UHF algebra.

Pick increasing sequences of finite sets $\mathcal F_{A, 1}\subset \mathcal F_{A, 2} \subset \cdots \subset A$ and $\mathcal F_{B, 1}\subset \mathcal F_{B, 2} \subset \cdots \subset B$ with dense unions; pick $\eps_1>\eps_2>\cdots>0$ with $\sum\eps_n =1$. 

Apply Theorem \ref{uniq-thm} to $A$ with respect to $\mathcal F_{A, 1}$ and $\eps_1$, one obtains $\mathcal G_{A_1}, \mathcal H_{A, 1}$, $\delta_{A, 1}, \sigma_{A, 1}>0$.

By Corollary \ref{decp}, $A\in\mathrm{TA}\mathcal R_z$, there is a Razak algebra $S\subset A$ and $\theta_1: A \to S$ which is $\mathcal G_{A, 1}$-$\delta_{A, 1}$-multiplicative map such that
$$| \tau(\theta(a)) - \tau(a) | < \frac{\sigma_{A, 1}}{2},\quad a\in \mathcal H_{A, 1},\ \tau\in\mathrm{T}^+A,\ \|\tau\|=1.$$

Consider the Cuntz group map $[\xi]\circ[\iota]: \mathrm{Cu}(S)\to\mathrm{Cu}(B)$, where $\iota: S\to A$ is the embedding map. Since $B$ has stable rank one, there is a homomorphism $\phi'_1: S\to B$ such that 
$$[\phi'_1] = [\xi]\circ[\iota].$$ 
Therefore,
$$\tau(\phi'_1(a))=\xi(\tau)(a),\quad a\in S,\ \tau\in\mathrm{T}^+B,$$
and hence
\begin{equation*}
\tau(\phi'_1(\theta(a))) = \xi(\tau)(\theta(a))\approx_{\frac{\sigma_{A, 1}}{2}} \xi(\tau)(a),\quad a\in\mathcal H_{A, 1},\ \tau\in\mathrm{T}^+B,\ \|\tau\|=1.
\end{equation*}
Denote by $\phi_1=\phi'_1\circ \theta: A\to B$, and one has
\begin{equation}\label{tr-approx-001}
| \tau(\phi_1(a)) - \xi(\tau)(a) | < \frac{\sigma_{A, 1}}{2}, \quad a\in\mathcal H_{A, 1},\ \tau\in\mathrm{T}^+B,\ \|\tau\|=1.
\end{equation}

Consider the C*-algebra $B$. Applying Theorem \ref{uniq-thm} to $B$ with respect to $\mathcal F_{B, 1}$ and $\eps_1$, one obtains $\mathcal G_{B, 1}$, $\mathcal H_{B, 1}$, $\delta_{B, 1}>0$ and $\sigma_{B, 1}>0$.

Since $B\in\mathrm{TA}\mathcal R_z$, with \eqref{tr-approx-001}, the same argument as above produces a completely positive contraction $\psi_1: B \to A$ such that $\psi_1$ is $\mathcal G_{B, 1}$-$\delta_{B, 1}$-multiplicative, $\psi_1\circ\phi_1$ is $\mathcal G_{A, 1}$-$\delta_{A, 1}$-multiplicative,
\begin{equation}\label{match-tr-1}
| \tau(\psi_1\circ\phi_1(a)) - \tau(a) | < \sigma_{A, 1},\quad a\in\mathcal H_{A, 1}, \ \tau\in\mathrm{T}^+A,\ \|\tau\|=1,
\end{equation} 
and
\begin{equation}\label{tr-approx-002}
| \tau(\psi_1(a)) - \xi^{-1}(\tau)(a) | < \frac{\sigma_{B, 1}}{2},\quad a\in\mathcal H_{B, 1}, \ \tau\in\mathrm{T}^+A,\ \|\tau\|=1.
\end{equation}


By \eqref{match-tr-1}, it following from Theorem \ref{uniq-thm} that there is a unitary $u\in \tilde{A}\otimes Q$ such that
$$\|u^*\psi_1(\phi_1(a))u - a \| <\eps_1,\quad a\in\mathcal F_{A, 1}.$$
Replace $\psi_1$ by $\mathrm{ad}(u)\circ\psi_1: B\to A$, one has that
$$\|\psi_1(\phi_1(a)) - a \| <\eps_1,\quad a\in\mathcal F_{A, 1};$$
that is, the following diagram commute up to $\eps_1$ on $\mathcal F_{A, 1}$:
\begin{displaymath}
\xymatrix{
A \ar@{=}[r] \ar[d]^{\phi_1} & A \\
B \ar[ur]_{\psi_1}& 
}.
\end{displaymath}

Applying Theorem \ref{uniq-thm} again to $A$ on $\mathcal F_{A, 2}$ and $\eps_2$, one obtains $\mathcal G_{A, 2}, \mathcal H_{A, 2}$, $\delta_{A, 2}, \sigma_{A, 2}>0$.

Using the fact that $A\in\mathrm{TA}\mathcal R_z$ again and with \eqref{tr-approx-002}, one obtains a completely positive contraction $\phi_2: A \to B$ such that $\psi_2$ is $\mathcal G_{A, 2}$-$\delta_{A, 2}$-multiplicative, $\phi_2\circ\psi_1$ is $\mathcal G_{B, 1}$-$\delta_{B, 1}$-multiplicative,
\begin{equation}\label{match-tr-2}
| \tau(\phi_2\circ\psi_1(a)) - \tau(a) | < \sigma_{B, 1},\quad a\in\mathcal H_{B, 1}, \ \tau\in\mathrm{T}^+B,\ \|\tau\|=1,
\end{equation} 
and
$$| \tau(\psi_2(a)) - \xi(\tau)(a) | < \frac{\sigma_{A, 2}}{2},\quad a\in\mathcal H_{A, 2}, \ \tau\in\mathrm{T}^+B,\ \|\tau\|=1.$$

%
By \eqref{match-tr-2}, it following from Theorem \ref{uniq-thm} that there is a unitary $w\in \tilde{B}\otimes Q$ such that
$$\|w^*\phi_2(\psi_1(a))w - a\|<\eps_1,\quad a\in\mathcal F_{B, 1}.$$
Replace $\phi_2$ by $\mathrm{ad}(w)\circ\phi_2: A\to B$, one has that
$$\|\phi_2(\psi_1(a)) - a \| <\eps_1,\quad a\in\mathcal F_{B, 1}.$$
Thus, we have the following approximate commutative diagram:
\begin{displaymath}
\xymatrix{
A \ar@{=}[r] \ar[d]^{\phi_1} & A \ar[d]^{\phi_2} \\
B \ar[ur]_{\psi_1} \ar@{=}[r]& B 
}.
\end{displaymath}

Repeat this procedure, one has the diagram
\begin{displaymath}
\xymatrix{
A \ar@{=}[r] \ar[d]^{\phi_1}& A \ar@{=}[r] \ar[d]^{\phi_2} & A \ar@{=}[r] \ar[d]^{\phi_3} & \cdots \ar@{=}[r]& A \\
B \ar@{=}[r] \ar[ur]_{\psi_1}& B \ar@{=}[r] \ar[ur]_{\psi_2} & B \ar@{=}[r] \ar[ur]_{\psi_3} &\cdots\ar@{=}[r] & B 
}
\end{displaymath}
such that
$$\|\psi_n(\phi_n(a)) - a \| < \eps_n,\quad a\in\mathcal F_{A, n},$$
and
$$\|\phi_{n+1}(\psi_n(a)) - a \| < \eps_n,\quad a\in\mathcal F_{C, n},$$
and hence $A\cong B$, as desired.
\end{proof}

\begin{thm}\label{clas-thm}
The class of simple separable KK-contractible C*-algebras with finite nuclear dimension is classifiable. If $\mathrm T^+A\neq \{0\}$, the invariant is $(\mathrm T^+A, \Sigma_A)$, and in this case,  any a such C*-algebra $A$ a simple inductive limits of Razak algebras. 

If $\mathrm T^+A= \{0\}$, then the invariant is the Murray-von Neumann semigroup together with the class of the projections of the algebra.
\end{thm}
\begin{proof}

If $\mathrm{T}^+A= \{0\}$, then $A$ is purely infinite. It follows from the classification of  Kirchberg and Phillips that $A$  is stably isomorphic to $\mathcal{O}_2$ and is classified by the Murray-von Neumann semigroup together with the class of the projections of the algebra. 

Assume $\mathrm{T}^+A\neq \{0\}$. 
Let $A, B$ be two such C*-algebras satisfying $$(\mathrm{T}^+A, \Sigma A) \cong (\mathrm{T}^+B, \Sigma B).$$ 

Consider $A\otimes\mathcal K$ and $B\otimes\mathcal K$, and note that the embeddings $A\to A\otimes\mathcal K$ and $B\to B\otimes \mathcal K$ induce homeomorphisms $\mathrm{T}^+(A\otimes\mathcal K)\cong \mathrm{T}^+A$ and $\mathrm{T}^+(B\otimes\mathcal K)\cong \mathrm{T}^+B$ respectively. Thus, one has that $\mathrm{T}^+(A\otimes\mathcal K)\cong \mathrm{T}^+(B\otimes\mathcal K)\cong C$ for a topological cone $C$ with a base $\Delta$ which is a compact metrizable Choquet simplex.

Consider the continuous affine function which is constant $1$ on $\Delta$, and denote $h_A\in A\otimes\mathcal K$ and $h_B\in B\otimes\mathcal K$ the positive elements which the Cuntz semigroup class is the given continuous affine function, respectively. Then one has that $\overline{h_A(A\otimes\mathcal K)h_A}\in \mathcal R$ and $\overline{h_B(B\otimes\mathcal K)h_B}\in\mathcal R$, and $\overline{h_A(A\otimes\mathcal K)h_A}\in \mathcal R$ and $\overline{h_B(B\otimes\mathcal K)h_B}\in\mathcal R$ has the same Elliott invariants; 
hence, by Theorem \ref{clas-thm-R}, $\overline{h_A(A\otimes\mathcal K)h_A}\cong \overline{h_B(B\otimes\mathcal K)h_B}$, and by Brown's Theorem, $A\otimes\mathcal K\cong B\otimes\mathcal K$.

Thus, one may identify $A=\overline{aDa}$ and $B=\overline{bDb}$ where $D$ a stable C*-algebra (with stable rank one). Since $\Sigma_A\cong\Sigma_B$, one has that $$\|\tau|_A\|=\|\tau|_B\|,\quad \tau\in\mathrm{T}^+D.$$
On the other hand, $$\|\tau|_A\| = \mathrm{d}_\tau(a)\quad\mathrm{and}\quad \|\tau|_B\| = \mathrm{d}_\tau(b).$$
Therefore
$$\mathrm{d}_\tau(a) = \mathrm{d}_\tau(b),\quad  \tau\in\mathrm{T}^+D,$$
and hence $[a]=[b]\in\mathrm{Cu}(D)$. Since $D$ has stable rank one, it follows from Theorem 3 of \cite{CEI-CuntzSG} that $A\cong B$.

It follows from \cite{Tsang-W} that the class of simple inductive limit of Razak algebras exhausts all possible values of the invariant. Thus, the class of simple, separable, finite nuclear dimension, KK-contractible, finite C*-algebras coincides with the class of simple inductive limits of Razak algebras.
\end{proof}

\begin{rem}
The invariants of Theorem \ref{clas-thm} in the finite case and infinite case to the semigroup can be unifies as
$$(\mathrm{V}(A)\setminus\{0\})\sqcup \mathrm{LAff}\mathrm{T^+A},$$
where $\mathrm{V}(A)$ is the Murray-von Neumann semigroup, $\mathrm{LAff}\mathrm{T^+A}$ is the semigroup of positive lower semicontinuous affine functions on $\mathrm T^+A$ arising from the Cuntz semigroup of $A$, together with the class of the algebra in it. If $\mathrm{T}^+A\neq \{0\}$, then this invariant is isomorphic to the Cuntz semigroup of $A$ together with the class of the algebra  (assuming $A$ is simple, separable, nuclear, and Jiang-Su stable); if $\mathrm{T}^+A = \{0\}$, then this invariant is the Murray-von Neumann semigroup together with the class of the projections of the algebra.
\end{rem}

\begin{cor}\label{clas-hom}
Let $A, B$ be simple separable KK-contractible C*-algebras with finite nuclear dimension. If there is a homomorphism $\xi: (\mathrm{T}^+B, \Sigma B) \to (\mathrm{T}^+A, \Sigma A)$, then there is a C*-algebra homomorphism $\phi: A\cong B$ such that $\phi_*=\xi$.
\end{cor}
\begin{proof}
This follows from the classification of Razak algebras (\cite{Razak-W}).
\end{proof}

Consider the C*-algebra $\mathcal W$, the simple inductive limit of Razak algebras with $$(\mathrm T^+\mathcal W, \Sigma\mathcal W) \cong ([0, +\infty), [0, 1]).$$ 
\begin{cor}\label{cor-WW}
Let $A$ be a simple separable C*-algebra such that $A\otimes\mathcal Z$ has finite nuclear dimension. Then the C*-algebra $A\otimes \mathcal W$ is classifiable. In particular, $\mathcal W\otimes \mathcal W \cong \mathcal W$.
\end{cor}

\begin{proof}
By Lemma \ref{KK-Kun}, $A\otimes \mathcal W$ is KK-contractible. Note that $$A\otimes\mathcal W\cong A\otimes\mathcal (Z\otimes\mathcal W)\cong (A\otimes\mathcal Z)\otimes\mathcal W.$$ Then $A\otimes \mathcal W$ has finite nuclear dimension, and the corollary follows from Theorem \ref{clas-thm}.
\end{proof}

\bibliographystyle{amsalpha}
\bibliography{operator_algebras}

\end{document}